\title{On the scaling limit of finite vertex transitive graphs with large diameter}
\author{Itai Benjamini \and Hilary Finucane \and Romain Tessera\footnote{Supported by ANR-09-BLAN-0059  and  ANR-10-BLAN 0116.} }
\documentclass[12pt]{article}

\usepackage{amssymb,amsmath,amsthm}
\usepackage{verbatim}
\usepackage{fullpage}
\usepackage[all]{xy}
\newtheorem{theoremintro}{Theorem}
\newtheorem*{thmm}{Theorem}
\newtheorem{theorem}{Theorem}[subsection]
\newtheorem{lemma}[theorem]{Lemma}
\newtheorem{proposition}[theorem]{Proposition}
\newtheorem{corollary}[theorem]{Corollary}

\newtheorem{conjecture}[theorem]{Conjecture}
\newtheorem{remark}[theorem]{Remark}

\newenvironment{definition}[1][Definition:]{\begin{trivlist}\item[\hskip \labelsep {\bfseries #1}]}{\end{trivlist}}

\newcommand{\n}{\mathfrak{n}}
\newcommand{\m}{\mathfrak{m}}

\newcommand{\PP}{\mathcal{P}}
\newcommand{\N}{\mathbb{N}}

\newcommand{\R}{\mathbb{R}}

\newcommand{\Z}{\mathbb{Z}}
\newcommand{\SL}{\textnormal{SL}}

\newcommand{\eps}{\varepsilon}

\def\diam{\mathop{\mathrm{diam}}}
\def\area{\mathop{\mathrm{area}}}

\begin{document}
\maketitle

\begin{abstract}
Let $(X_n)$ be an unbounded sequence of finite, connected, vertex transitive graphs such that $ |X_n | = O( \diam(X_n)^q)$ for some $q>0$. We show that up to taking a subsequence, and after rescaling by the diameter, the sequence $(X_n)$ converges in the Gromov Hausdorff distance to some finite dimensional torus equipped with some invariant Finsler metric. The proof relies on a recent quantitative version of Gromov's theorem on groups with polynomial growth obtained by Breuillard, Green and Tao. 
 If $X_n$ is only roughly transitive and $|X_n| = O\bigl({\diam(X_n)^{\delta}}\bigr)$ for $\delta > 1$ sufficiently small, we prove, this time by elementary means, that $(X_n)$ converges to a circle.
\end{abstract}
\tableofcontents
\section{Introduction}

\subsection{Scaling limits of transitive graphs}
A graph $X$ is {\em vertex transitive} if for any two vertices $u$ and $v$ in $X$, there is an automorphism of $X$ mapping $u$ to $v$.
Let $(X_n)$ be a sequence of finite, connected, vertex transitive graphs. Rescale the length of the edges of $X_n$ by the inverse of the graph's diameter so that the resulting metric space $X'_n$ has diameter $1$.  A metric space $\cal M$ is the {\em scaling limit} of $(X_n)$ if $(X'_n)$ converges to $\cal M$ in the Gromov Hausdorff distance. See e.g. \cite{BH, BBI, Gromov} for background on scaling limits and Gromov Hausdorff distance. 

In this paper we address the following questions: when is the scaling limit of such a sequence a compact homogeneous manifold? And in that case, what can be said  about the limit manifold? By a compact homogeneous manifold, we mean a compact topological manifold $A$ equipped with a geodesic distance (not necessarily Riemannian), such that the isometry group acts transitively on $A$.  

As nicely observed by Gelander in \cite{G}, the answer to the second question is rather restrictive. Using an old theorem of Turing \cite{Turing} (later rediscovered by Kazhdan \cite{Ka}) along with some result about equivariant Gromov-Hausdorff convergence (essentially \cite[Proposition]{FY}, 
see \S \ref{section:HG}), he obtains that any compact homogeneous manifold approximated by finite homogeneous metric spaces must be a torus. More generally, Gelander proves that a homogeneous compact metric space can be approximated by finite homogeneous metric spaces if and only if it admits a compact group of isometries acting transitively whose connected component is abelian. Note that it may be counter-intuitive to realize that a sphere cannot be approximated by a vertex-transitive graph!

In this paper, we address both questions. Recall that a Finsler metric on $T^n=\R^n/\Z^n$ is a geodesic metric induced by a norm on $\R^n$. Our main result is the following:

\begin{theoremintro}\label{theorem:Main}
Let $(X_n)$ be a sequence of vertex transitive graphs such that $|X_n| \rightarrow \infty$ and $|X_n| = O(\diam(X_n)^q)$. Then $(X_n)$ has a subsequence whose scaling limit is a torus of finite dimension  equipped with an invariant Finsler metric.
\end{theoremintro}
The proof of Theorem \ref{theorem:Main} makes crucial use of a recent quantitative version of Gromov's theorem obtained by Breuillard, Green and Tao \cite{BGT}, allowing us to reduce the problem from vertex transitive graphs to Cayley graphs of nilpotent groups. 

When we assume moreover that the degree of the vertex transitive graphs in our sequence is bounded, we can be more specific about the limiting metric. Let us say that a Finsler metric is polyhedral if the unit ball for the corresponding norm is a polyhedron. 

\begin{theoremintro}\label{theorem:Main'}
Let $(X_n)$ be a sequence of vertex transitive graphs with bounded degree such that $|X_n| \rightarrow \infty$ and $|X_n| = O(\diam(X_n)^q)$. Then $(X_n)$ has a subsequence whose scaling limit is a torus of dimension  $\leq q$ equipped with an invariant polyhedral Finsler metric.
\end{theoremintro}

This stronger conclusion does not hold  without the bound on the degree: indeed, it is not hard to see that any Finsler metric on a torus of dimension $d$ can be obtained as a scaling limit of a sequence of Cayley graphs (not necessarily with bounded degree) of finite abelian  groups. 
In a previous version, we claimed to have the same bound on the dimension in Theorem \ref {theorem:Main}, but we later realized that our proof contained a gap.

In the course of proving that the limiting Finsler metric is polyhedral in  Theorem \ref{theorem:Main'} we shall prove a result which seems to be of independent interest: the stabilizers of a faithful transitive action of a step $l$, rank $r$ nilpotent group on a graph of degree $d$ have size bounded by some function of $l,r$ and $d$ (see Lemma \ref{lem:Torsion Nilpotent}).

\

To finish this section, let us illustrate Theorem~\ref{theorem:Main'} on a few examples.
\begin{itemize}
\item First, it is easy to see that the scaling limit of the sequence of Cayley graphs $(\Z/n\Z,\{\pm1\})$ is the circle. 

\item More generally, letting $(e_1,\ldots e_k)$ denote the canonical basis of $\Z^k$, the scaling limit of $((\Z/n\Z)^k,\{\pm e_1,\ldots, e_k\})$ is the $k$-dimensional torus equipped with the Finsler metric associated to the $\ell^1$-norm on $\R^k$. 

\item To illustrate the fact that the generating set (not only the group) plays a role, observe that the scaling limit of $(\Z/n^{k}\Z, \{\pm 1, \pm n,\ldots,\pm n^{k-1}\})$ is isometric the $k$-dimensional torus with the $\ell^1$-metric.

\item Finally let us give a more interesting example. Given a ring $A$, one can consider the Heisenberg group $H(A)$ of $3$ by $3$ upper unipotent matrices with coefficients in $A$.
Now let $X_n$ be the Cayley graph of the group $H(\Z/n\Z)$ equipped with the finite generating set consisting of the $3$ elementary unipotent matrices and their inverses.  The cardinality of $X_n$ equals $n^3$ and easy calculations show that its diameter is in $\Theta(n)$. For every $n$,
we have a central exact sequence
$$1\to \Z/n\Z\to H(\Z/n\Z)\to (\Z/n\Z)^2\to 1,$$
whose center is quadratically distorted. In other words, the projection from $X_n$ to the Cayley graph of $(\Z/n\Z)^2$ has fibers of diameter $\simeq \sqrt n$. It follows that the rescaled sequence $(X'_n)$ converges to the $2$-torus $\R^2/\Z^2$ with its $\ell^1$ Finsler metric. 
\end{itemize}

\subsection{Toward an elementary proof of Theorem~\ref{theorem:Main}}

We think an interesting and potentially very challenging open question is to provide a proof of Theorem~\ref{theorem:Main} that does not rely on the results of Breuillard, Green, and Tao (at least for ``small" values of $q$). In Section~\ref{section.bis}, we present elementary proofs of two results similar to the $q = 2$ case of Theorem~\ref{theorem:Main}. In the first result, the requirement that $X_n$ be vertex transitive is weakened, but the assumption on the diameter is strengthened. In the second, the requirement of vertex transitivity is restored, but the assumption on the diameter is weaker than in the first (though still stronger than in Theorem~\ref{theorem:Main}).

Recall that for $C\geq 1$ and $K\geq 0$, a $(C,K)$-quasi-isometry between two metric spaces $X$ and $Y$ is a map $f: X\to Y$ such that $$C^{-1}d(x,y)-K\leq d(f(x),f(y))\leq Cd(x,y)+K,$$ and such that every $y\in Y$ is at distance at most $K$ from the range of $f$. Let us say that a metric space $X$ is $(C,K)$-roughly transitive if for every pair of points $x,y\in X$ there is a $(C,K)$-quasi-isometry of $X$ sending $x$ to $y$. Let us call a family of metric spaces roughly transitive if there exist some $C\geq 1$ and $K \geq 0$ such that each member of the family is $(C,K)$-roughly transitive. 

In Section~\ref{section.bis}, we provide an elementary proof of the following theorem.

\begin{theoremintro}\label{theorem.bis} Suppose $(X_n)$ is a  roughly transitive sequence of finite graphs such that $|X_n| \rightarrow \infty$. There exists a constant $\delta>1$ such that if
$$
|X_n | = O(\diam(X_n)^{\delta}),
$$
then the scaling limit of $(X_n)$ is $S^1$.
\end{theoremintro}

Modifying slightly the proof of Theorem \ref{theorem.bis}, one can prove that for an infinite, roughly transitive graph $X$ with bounded degree, there exists $R>0$, $C>0$ and $\delta>0$ (depending only on the rough transitivity constants and the degree) such that if for all $R'\geq R$, and all $x\in X$, $$|B(x,R')|\leq CR'^{1+\delta},$$ then $X$ is quasi-isometric to $\R$. Even for vertex-transitive graphs, this provides a new elementary proof (compare \cite{DM}).

Our second result assumes vertex transitivity, but replaces $\delta > 1$ with an explicit number:

\begin{theoremintro}\label{theorem.main} Suppose $X_n$ are vertex transitive graphs with $|X_n| \rightarrow \infty$ and $$|X_n| = o(\diam(X_n))^{2-\frac{1}{\log_3(4)}}.$$ Then the scaling limit of $(X_n)$ is $S^1$. \end{theoremintro}

In Section \ref{sec:counterexamples}, we shall see that Theorem \ref{theorem.bis} has counterexamples for any $\delta> 2$, namely we shall prove

\begin{theoremintro}
For every positive sequence $u(n)$ such that $u(n)/n^2\to \infty$, there exists a roughly transitive sequence of finite graphs $X_n$ such that $|X_n|\leq u(\diam(X_n)$ for large $n$, and such that no subsequence of $X_n$ has a scaling limit. 
\end{theoremintro}

\subsection{Motivations and further questions}
Theorem \ref{theorem:Main} can be seen as a first result towards understanding the large-scale geometry of finite vertex-transitive graphs with ``large diameter". 
There is a vast literature regarding the stability of analytic and geometric properties (such as Cheeger constant, spectrum of the Laplacian...) under Gromov-Hausdorff convergence: this has especially been studied in the context of Riemannian manifolds \cite{CC,Fu1,Fu2}, but also for more general sequences of compact metric spaces \cite[\S 5.3]{K}. However, these results always rely on very strong conditions on the local geometry: the most common one being a uniform lower bound on the Ricci curvature (which makes sense for Riemannian manifolds), or at least a uniform doubling condition and some (also uniform) local Poincar\'e inequality.   
Unfortunately, we cannot apply these results in our setting. Indeed, it is easy to come up with examples where the local geometry is --so to speak-- quite arbitrary: starting with any sequence $Y_n$ of finite vertex-transitive graphs, one can consider the direct product $X_n=Y_n\times C_{k_n}$, where $C_{k_n}$ is the cyclic graph of size $k_n$. For every $\eps>0$, if $k_n$ is sufficiently large, one can make sure that $|X_n|\leq \diam(X_n)^{1+\eps}$. On the other hand, if for instance the girth $g_n$ of $Y_n$ tends to infinity (while its degree is $\geq 3$) the resulting graphs $X_n$ fail to be doubling at scales smaller than $g_n$. 
 
Nevertheless, this does not mean that no quantitative estimate can be obtained. But in order to deduce them from Theorem \ref{theorem:Main}, some additional work is needed. In a follow-up of this paper, we plan to exploit the proof of Theorem \ref{theorem:Main} to study quantitatively  graphs with large diameter from the point of view of their Cheeger constant, spectral gap, mixing time, harmonic measures (see \cite{BY}), convergence of the simple random walk to the Brownian motion on the scaling limit, and existence of a limiting shape for first passage percolation on the scaling limit.

\subsection{Organization}

Section~\ref{section:preliminaries} is devoted to preliminary lemmas.  Section~\ref{section:MainTheorem} is devoted to the proofs of Theorems \ref{theorem:Main} and \ref{theorem:Main'}. More precisely, Section~\ref{section:nilpotent} contains the reduction to virtually nilpotent groups. The proof of Theorem \ref{theorem:Main} is completed in subsection \ref{section:Main}. Section~\ref{section:abelian} contains the reduction to abelian groups, allowing us to bound the dimension of the limiting torus in Section~\ref{section:dimension}. We prove that the limiting metric in Theorem  \ref{theorem:Main'} is polyhedral in Section~\ref{section:pansu}, using Pansu's theorem, completing the proof of Theorem~\ref{theorem:Main'}.
In Section~\ref{section.bis}, we prove Theorems~\ref{theorem.bis} and \ref{theorem.main}. In Section \ref{sec:counterexamples}, we construct examples of sequences of roughly transitive graphs with polynomial growth for which no subsequence has a scaling limit. Finally in the last section, we list several open questions and formulate a conjectural version of Theorem \ref{theorem:Main} for infinite graphs. We also prove an easy fact about convergence in the sense of marked groups. 

\

\noindent{\bf Note to the reader:}
For those only interested in the proof of  Theorem \ref{theorem:Main}, the proof, relatively short, is covered by Sections \ref{section:nilpotent} and \ref{section:Main}, which punctually refer to the preliminaries in Sections \ref{section:HG} and \ref{sectionPrelim:compact}.

\

\noindent{\bf Notation:} In the sequel, we will often use the following convenient notation:
We will say that a number is in $O_{x_1, \ldots, x_n}(1)$ if it is bounded by some fixed function of $x_1, \ldots , x_n$.

%
%
%
%
%
%
%
%

\section{Preliminaries}
\label{section:preliminaries}

\subsection{Gromov-Hausdorff limits, ultralimits and equivariance}\label{section:HG}
\label{subsection:prelim1}

\begin{definition}\label{GH}
Given a  sequence $X_n$ of compact metric spaces we will say that $X_n$ GH-converges to $X$ if the $X_n$ have bounded diameter and if there exist maps $\phi_n:X_n\to X$ such that for all $\eps$, then for  $n$ large enough, 
\begin{itemize}
\item every point of $X$ is at $\eps$-distance of a point of $\phi_n(X_n)$;
\item $(1-\eps)d(x,y)-\eps\leq d(\phi_n(x),\phi_n(y))\leq  (1+\eps)d(x,y)+\eps$ for all $x,y\in X_n$.
 \end{itemize}
\end{definition}
Recall that GH-convergence extends naturally to (not necessarily compact) locally compact pointed metric spaces (see \cite[Section 3]{Gromov}). 
\begin{definition}  
Given a  sequence $(X_n,o_n)$ of locally compact pointed metric spaces, $(X_n,o_n)$ is said to converge to the locally compact pointed metric space $(X,o)$ if for every $R>0$, the sequence of balls $B(o_n,R)$ GH-converges to $B(o,R)$.
\end{definition}

The following lemma will play a central role in this paper.
\begin{lemma}\label{lemPrelim:compact}
(Gromov's compactness criterion, \cite[Theorem 5.41]{BH}) A sequence of compact metric spaces $(X_n)$ is relatively Gromov Hausdorff compact if and only if the $X_n$'s have bounded diameter, and are ``equi-relatively compact": for every $\eps>0$, there exists $N\in \N$ such that for all $n\in \N$, $X_n$ can be covered by at most $N$ balls of radius $\eps.$
\end{lemma}

In \cite[Section 3]{Gromov}, Gromov discusses functoriality of these notions of convergence. In particular, he mentions a notion of equivariant GH-convergence. In \cite[\S 3]{FY}, Fukaya and Yamaguchi provide a precise notion of GH-convergence for sequences  of triplets $(X_n,G_n,o_n)$, where $(X_n,o_n)$ are locally compact pointed metric spaces, and $G_n$ is a group of isometry of $X_n$. 
 They prove the important fact that if $(X_n,o_n)$ GH-converges to $(X,o)$ and if $G_n$ is a group of isometries of $X_n$ acting transitively, then in some suitable sense, a subsequence of the triplet $(X_n,G_n,o_n)$ converges to $(X,G,o)$ where $G$ is a subgroup of isometries of $X$ acting transitively. We shall prove an analogue of this statement, although using the convenient language of ultralimits. 
   
 First recall that an ultrafilter (see \cite{Com}) is a map from $\omega:\PP(\N)\to \{0,1\}$, such that $\omega(\N)=1$, and which is ``additive" in the sense that $\omega(A\cup B)=\omega(A)+\omega(B)$ for all $A$ and $B$ disjoint subsets of $\N$. 
Ultrafilters are used to ``force" convergence of bounded sequences of real numbers. Namely, given such a sequence $a_n$, its limit is the only real number $a$ such that for every $\eps>0$ the subset $A$ of $\N$ of integers $n$ such that $|a_n-a|<\eps$ satisfies $\omega(A)=1$. In this case, we denote $\lim_{\omega} a_n=a$. An ultrafilter is called non-principal if it vanishes on finite subsets of $\N$. Non-principal ultrafilters are known to exist but this requires the axiom of choice. In the sequel, let us fix some non-principal ultrafilter $\omega$. 

\begin{definition}
Given a sequence of pointed metric spaces $(X_n,o_n)$, its ultralimit with respect to $\omega$ is the quotient of 
$$\{(x_n)\in \Pi_n X_n, \; \exists C>0,\; \forall n, \; d(x_n,o_n)\leq C\}$$
 by the equivalence relation $x_n\sim y_n$ if $\lim_{\omega}d(x_n,y_n)= 0$. It is equipped with a distance defined by $d((x_n),(y_n))=\lim_{\omega}d(x_n,y_n).$
\end{definition}
It is a basic fact that a sequence $a_n\in \R$ converging  to $a$ satisfies $\lim_{\omega}a_n=a$. This fact actually extends to ultralimits of metric spaces:

\begin{lemma}\label{lemPrelim:ultra}\cite[Exercice 5.52]{BH} If a sequence of pointed metric spaces converges in the pointed GH sense to $X$, then its ultralimit with respect to $\omega$ is isometric to $X$. 
\end{lemma}

Let $(X_n,d_n)$ be a  sequence of compact metric spaces of diameter $\leq 1$, and for every $n$, let $G_n$ be a group of isometries acting transitively on $X_n$. Let $\hat{d}_n$ be the bi-invariant distance on $G_n$ defined by
$$\hat{d}_n(f,g)=\sup_{x\in X_n}d_n(f(x),g(x)).$$
Recall that a distance $d$ on a group $G$ is bi-invariant if it is invariant under both left and right translations: i.e.\ for all $g_1,g_2,g,h\in G$, $d(g_1gg_2,g_1hg_2)=d(g,h)$. 
\begin{lemma}\label{lem:ultralim}
Suppose $X_n$ is a sequence of homogeneous metric spaces of diameter $\leq 1$, and let $G_n$ be a group of isometries acting transitively on $X_n$.  Let $X$ denote the $\omega$-ultralimit of $(X_n)$. Then the $\omega$-ultralimit $(G,\hat{d})$ of $(G_n,\hat{d}_n)$ naturally identifies with a group of isometries acting transitively on $X$, where $\hat{d}$ is the bi-invariant distance on $G$ defined by
$\hat{d}(f,g)=\max_{x\in X}d(f(x),g(x)).$
\end{lemma}
\begin{proof}
The fact that $\hat{d}_n$ is bi-invariant implies that  the group law on $G_n$ ``$\omega$-converges" and defines a group law on $G$. Namely, let $a,b\in G$ and let  $(a_n)$, $(b_n)$ converge to $a$ and $b$, respectively. Let us check that the limit ``$a^{-1}b$" of the sequence $(a_n^{-1}b_n)$ does not depend on the choice of $(a_n)$ and $(b_n)$ and therefore defines our group law on $G$. Suppose $\alpha_n$ and $\beta_n$ also $\omega$-converge to $a$ and $b$. Using that  $\tilde{d}_n$ is bi-invariant, we obtain
\begin{eqnarray*}
d(a_n^{-1}b_n,\alpha_n^{-1}\beta_n) &= &d(b_n,a_n\alpha_n^{-1}\beta_n)\\
                                                   & \leq & d(b_n, \beta_n)+d(\beta_n,a_n\alpha_n^{-1}\beta_n)\\
                                                   & =  &  d(b_n, \beta_n) + d(a_n,\alpha_n)
\end{eqnarray*}
whose $\omega$-limit is zero by assumption. The action of $G$ on $X$ is defined similarly: for every $(g_n)$ converging to $g\in G$ and every $(x_n)$ converging to $x$ in $X$ we define $gx$ as the limit of $(g_nx_n)$. One can also check that this limit does not depend on the choices of $(g_n)$ and $(x_n)$. 
All the remaining statements of the lemma are straightforward and left to the reader.
\end{proof}

Given a reduced word $w(x_1,\ldots,x_k)$ in the free group with $k$ generators $\langle x_1,\ldots, x_k\rangle$, we say that a group $G$ satisfies the group law $w=1$, if $w(g_1,\ldots,g_k)=1$ for all $g_1,\ldots,g_k\in G.$ Clearly, if the sequence of groups $(G_n,\hat{d}_n)$ from the previous lemma satisfies a group law, then so does its $\omega$-limit. In particular we deduce the following corollary (we refer to the beginning of \S \ref{subsection:prelim3} for the definition of $l$-step nilpotent groups).

\begin{corollary}\label{corPrelim}
Let $X_n$ be a sequence of homogeneous compact metric spaces of  diameter $\leq 1$, which GH-converges to $X$. For every $n$, let $G_n$ be a subgroup of the isometry group of $X_n$ acting transitively. Then $G=\lim_{\omega}(G_n,\hat{d}_n)$ identifies with some closed transitive subgroup of the isometry group of $X$. Moreover, if there exists $l$ and $i$ such that all $G_n$ are virtually nilpotent of step $\leq l$ and index $\leq i$, then so is $G$.
\end{corollary}

\noindent{\bf Extension to pointed GH-converging homogeneous metric spaces.}
Although Corollary \ref{corPrelim} is enough to deal with Theorem \ref{theorem:Main}, we need to extend it to locally compact metric spaces in order to obtain Theorem \ref{theorem:Main'}.

First, observe that if the pointed metric space $(X,d,o)$ is unbounded, the distance $\hat{d}$ on the isometry group $G$ of $X$ might take infinite values. So here, it will be convenient to replace $\hat{d}$ by a neighborhoods basis of the identity. Let us assume that $X$ is locally compact. If $G$ is equipped its usual (compact-open) topology for which it is a locally compact group, we can specify a basis of compact neighborhoods of the identity as follows: 
For every $R>0$, and every $\eps>0$
$$\Omega(R,\eps)=\{g\in G, \;\sup_{x\in B(o,R)}d(f(x),x)\leq \eps\}.$$ 

Now, given a sequence of pointed metric spaces $(X_n,o_n)$, and an ultrafilter $\omega$, we define the ultralimit $G$ of $(G_n)$ to be the quotient of 
$$\{(g_n)\in \Pi_n G_n, \; \exists C>0,\; \forall n, \; d(g_no_n,o_n)\leq C\}$$
by the equivalent relation $f_n\sim g_n$ if for every $R\in \N$, $\lim_{\omega}\sup_{x\in B(o_n,R)}d(f_n(x),g_n(x))=0$. For every $R\in \N$, define 
$\lim_{\omega}\Omega_n(R,\eps)$ to be the set of $(g_n)\in G$ such that $\omega$-almost surely, $g_n\in \Omega_n(R,\eps)$.
As in the proof of Lemma \ref{lem:ultralim}, one easily checks that the group law is well-defined at the limit (i.e.\ that the product and the inverse do not  depend on the sequences representing the elements), and that the ultralimit $G$ identifies with a closed subgroup of isometries of $X:=\lim_{\omega}(X_n,o_n)$. Also,  $\lim_{\omega}\Omega_n(R,\eps)$ naturally identifies with $\Omega(R,\eps)$ corresponding to this action. Moreover, if $G_n$ act transitively for all $n$, then so does $G$.   Observe that we get ``for free" that $G$ comes naturally with a topology that makes it a locally compact group.
 
Finally let us mention that Corollary \ref{corPrelim} extends to this setting as well.

\subsection{Locally compact connected homogeneous metric spaces}\label{sectionPrelim:compact}
\begin{lemma}\label{lemPrelimFaithful}
Let $G$ be a locally compact group acting faithfully, transitively and continuously on a locally compact, connected metric space of finite dimension. Then $G$ is a Lie group.
\end{lemma}
\begin{proof}
This is the main result of \cite{MZ}. However, if $G$ is compact, this can be seen as a easy consequence of a theorem of  Peter Weyl  \cite{PW} which says that any compact group is a projective limit of compact Lie groups.  Then the first corollary of \S 6.3 (p 243) in \cite{MZ} (whose proof is only a few lines) implies that $G$ is a compact Lie group.  
\end{proof}
 
\begin{lemma}\label{lemPrelimConnected}
Let $G$ be a Lie group acting transitively and continuously on a locally compact, connected space. Then the connected component of the identity, $G^0$,  acts transitively as well.
\end{lemma}
\begin{proof}  
Let $H$ be the stabilizer of a point $x$ in $X$, so that we can identify $X$ with the quotient $G/H$ equipped with the quotient topology. The projection map $\pi:G\to G/H$ being open, $\pi(G^0)= G^0H$ is an open subset of $G/H$. Let $G=\cup_i G^0g_iH$ be the decomposition of $G$ into double cosets. Since $G^0$ is normal in $G$, each double coset $G^0g_iH$ can be written as a left-translate of $\pi(G^0)$, namely: $g_iG^0H$. It follows that the complement of $\pi(G^0)$ is a union of left-translates of $\pi(G^0$), and therefore is open. In conclusion, $\pi(G^0)$ being open and closed, it equals $X$. Since $\pi(G^0)$ is the orbit of $x$ under the action of $G^0$, this implies that $G^0$ acts transitively on $X$.
\end{proof}

\begin{lemma}\label{lemPrelim:torus}
A virtually nilpotent connected Lie group $G$ is nilpotent and any compact subgroup is a central torus.  
\end{lemma}
\begin{proof}
This is of course well-known, but let us sketch its proof. The Lie algebra $\mathfrak{g}$ being nilpotent, the Lie group $\tilde{G}=\exp(\mathfrak{g})$ is nilpotent and simply connected (as the exponential map is a homeomorphism). In other words, $\tilde{G}$ is the universal cover of $G$, so it follows that $G$ is a quotient of $\tilde{G}$ by a discrete central subgroup.    
Let $K$ be a compact subgroup of $G$ and let $Z(G)$ be the center of $G$. Observe that $G/Z(G)$ is isomorphic to $\tilde{G}/Z(\tilde{G})=\exp(\mathfrak{g}/Z(\mathfrak{g}))$, and therefore is simply connected. In particular, it has no non-trivial compact subgroup. It follows that $K$ is contained in $Z(G)$. 
\end{proof}

We can now state the main result of this section.
\begin{theorem}\label{thmPrelim:PW}
Let $G$ be a virtually nilpotent locally compact group acting transitively by isometries on a locally compact geodesic metric space $X$ of finite dimension. Then $X$ is a connected nilpotent Lie group equipped with a left-invariant geodesic (actually a Carnot-Caratheodory) metric. Moreover, if $X$ is compact, then it is a torus equipped with a left-invariant Finsler metric.  
\end{theorem}

\begin{proof} 
The notion of Carnot-Caratheodory metric will be recalled in \S \ref {secPrelim:pansu}. 
Without loss of generality, we can assume that the action is faithful. 
By Lemma \ref{lemPrelimFaithful}, $G$ is a Lie group, and by Lemma \ref{lemPrelimConnected}, we can assume it is connected. Because $G$ acts transitively on $X$, we have $X=G/K$ where $K$ is the stabilizer of a point and a compact subgroup of $G$. By Lemma \ref{lemPrelim:torus}, $G$ is nilpotent and $K$ is central, and so $X=G/K$ is a connected nilpotent Lie group with a left-invariant geodesic metric. The fact that this metric is Carnot-Caratheodory is a consequence of the main result of \cite{B}.

Let us prove that in the compact case, the metric is Finsler. This is equivalent to the fact that any translation-invariant proper geodesic metric $d$ on $\R^m$ is associated to a norm.  This statement finally reduces to proving that balls with respect to $d$ are convex (the fact that they are symmetric resulting from translation-invariance). Here is a quick proof of this fact. 
Fix a Euclidean metric $d_e$ on $\R^m$. Fix some radius $r>0$ and some small number $\varepsilon>0$. Let $n$ be some large enough integer so that $(m+1)r/n\leq \varepsilon$. By the following lemma, $B_d(0,r)=B_d(0,r/n)^n$ is at GH-distance at most $\varepsilon$ from its convex hull. Since $\varepsilon$ was chosen arbitrarily small, this implies that $B(0,r)$ is convex, so we are done.
\end{proof}

\begin{lemma}\label{lem:convexhull}
Let $K$ be a compact symmetric subset of $\R^m$ containing $0$, of Euclidean diameter $\leq \varepsilon$, and let $\hat{K}$ be its convex hull. Then for every $n\in \N$,  the Gromov-Hausdorff distance relative to the Euclidean distance from $K^n=\{x\in \R^m, \; x=k_1+\ldots +k_n, \; k_i\in K\}$ to its convex hull is at most $(m+1)\varepsilon$.   
\end{lemma}
\begin{proof}
Clearly the convex hull of $K^n$ coincides with $\hat{K}^n$. Now let $x\in \hat{K}^n$, then $x=ny$, where $y\in \hat{K}$. Now $y$ can be written as a convex combination $y=t_0y_0+\ldots +t_my_m$ of $m+1$ elements of $K$. Write $nt_i=n_i+s_i$, where $s_i\in [0,1)$, and $n_i=[nt_i]$, so that $\sum_i n_i\leq n$ and $\sum_i s_i\leq m+1$.
Then, $x=ny=u+z$, where $u\in K^n$, and where $z\in \hat{K}^{m+1}$. This proves the lemma.   
\end{proof}

\subsection{Vertex transitive graphs}\label{sec:transitive}

The material of this section goes back to Abels \cite{Ab} and we refer to \cite{Ab,Lor} for the proofs.
Let $G$ be a group and let $S$ be a symmetric generating subset which is bi-$H$-invariant, i.e.\ $HSH=S$. Note that the Cayley graph $(G,S)$ is invariant by right-translations by $H$. The ``Cayley-Abels" graph associated to the data $(G,H,S)$ is defined as the quotient of $(G,S)$ by this action of $H$. Since left-translations by $G$ commute with right translations by $H$, this graph is $G$-transitive. More concretely, the vertex set of $(G,H,S)$ is $G/H$ and for every $g,g'\in G$, the two vertices $gH$ and $g'H$ are neighbors if and only if $g'=gs$ for some $s\in S.$

It turns out that all vertex-transitive graphs can be obtained in this way. Indeed, let $X$ be a $G$-transitive graph, and let $v_0$ be some vertex. Let $H$ be the stabilizer of $v_0$, and identify the vertex set with $G/H$ with $v_0$ corresponding to the trivial coset $H$, and let $\pi:G\to X$ be the projection.
Then $X=(G,H,S)$, where $S=\pi^{-1}(E(v_0))$, where $E(v_0)$ is the set of vertices joined by an edge to $v_0$ (observe that $E(v_0)$ contains $v_0$ if and only if $X$ has self-loops).
We deduce 

\begin{lemma}\label{lemPrelim:CA} For every vertex transitive graph $X$, every group $G$ acting transitively on $X$, and every vertex $x\in X$, there is a subgroup $H$ and bi-$H$-invariant subset $S$ of $G$ such that there is a $G$-equivariant graph isomorphism from $X$ to $(G,H,S)$ mapping $x$ to the trivial coset $H$. 
\end{lemma}

Because $gH$ and $g'H$ are neighbors in $(G,H,S)$ if and only if $g$ and $g'$ are neighbors in $(G,S)$, we have that, for $r \geq 2$,  $\pi(B_G(1,r))=B_X(v_0,r)$ and $\pi^{-1}(B_X(v_0,r))=B_G(1,r)$, where $\pi$ is the projection map. 
We deduce from this fact the following.
\begin{lemma}\label{lemPrelim:QI}
Let $X$ be a graph on which a group $G$ act transitively, let $x \in X$. Let $S$ be the generating subset of $G$ from Lemma \ref{lemPrelim:CA}. Then the projection $(G,S)\to X$ is a $(1,2)$-quasi-isometry.
\end{lemma}

\subsection{Finite index subgraphs}
\label{subsection:prelim2}
The following lemma will be used in various places.
\begin{lemma}\label{lem:subgraph}(Finite index subgraph)
Let $X$ be a connected graph of degree $d$, and $G$ a group acting transitively by isometries on its vertex set.
Let $G'<G$ be a subgroup of index $m<\infty$. Let $X'$ be a $G'$-orbit. Then $X'$ is the vertex set of some $G'$-invariant  graph that is $(O(m), O(m))$-QI to $X$, and whose degree is bounded by $d^{2m+1}$.
\end{lemma}
\begin{proof}
Denote by $[X']_k$ the $k$-neighborhood of $X'$ in $X$. Because $[X']_k$ is $G'$-invariant, it is a union of $G'$-orbits in $X$. If $[X']_{k-1} \neq X$, then $[X']_k$ contains at least one vertex not in $[X']_{k-1}$, and thus at least one $G'$-orbit not in $[X']_{k-1}$. So by induction, $[X']_k$ is either equal to $X$, or contains at least $k$ $G'$-orbits. Since $X$ is a union of $m$ $G'$-orbits, we must have $X=[X']_m$.

Define a $G'$-invariant graph on $X'$ by adding edges between two vertices of $X'$ if they are at distance at most $2m+1$ in $X$. Let $y,y'\in X'$. Since $d_X(y,y') \leq (2m+1) d_{X'}(y,y')$ by definition, it suffices to show that $d_{X'}(y,y') \leq d_X(y,y')$ in order to conclude that that $X'$ is $(2m+1, m)$-QI to $X$.

Let $y=x_0,\ldots,x_{v+1}=y'$ be a shortest path between $y$ and $y'$ in $X$. We will construct a path $y = y_0, \ldots , y_{v+1} =y'$ between $y$ and $y'$ in $X'$ as follows: for each $i=1\ldots v$, let $y_i$ be an element in $X'$ that is at distance at most $m$ from $x_i$ in $X$. Then the distance in $X$ between two consecutive $y_i$ is at most $2m+1$, so they are connected by an edge in $X'$. Thus, $y_0, \ldots, y_{v+1}$ is a path in $X'$, and $d_{X'}(y,y') \leq d_X(y,y')$.

\end{proof}

\subsection{Transitive actions of (virtually) nilpotent groups}\label{subsection:prelim3}

This subsection, culminating in Lemma \ref{lem:Torsion Nilpotent}, will only be required for the last statement of Proposition \ref{prop:NilpotentReduction}, which is needed for the proof of Theorem~\ref{theorem:Main'} but not for the proof of Theorem~\ref{theorem:Main}. We also believe that Lemma \ref{lem:Torsion Nilpotent} might be of independent interest. However, a reader interested only in the proof of Theorem \ref{theorem:Main} can directly jump to the next section.

First, some definitions. For group elements $y$ and $z$, let $[y,z]$ denote the commutator element $yzy^{-1}z^{-1}$. The commutator of $x$ with $[y,z]$ is thus denoted $[x,[y,z]]$, and an iterated commutator is denoted $[x_1, [x_2, \ldots , [x_{i-1},x_i]\ldots]]$. Given two subgroups $A$ and $B$ of the same group $G$, we shall denote by $[A,B]$ the subgroup generated by $[a,b]$ where $a\in A$ and $b\in B$. Let $C^j(G)$ be the descending central series of $G$, i.e.\ let $C^{0}(G)=G$, and $C^{j+1}(G)=[G,C^j(G)]$.  $G$ is $l$-step nilpotent if $C^{l} = \{1\}$ and $C^{l-1} \neq \{1\}$.
In the sequel, it will be convenient to use the following notation: 
$$[g_1,\ldots,g_i]:=[g_1, [g_2, \ldots , [g_{i-1},g_i]\ldots],$$
for $i\geq 2$ and $(g_1,\ldots,g_i)\in G$. Observe that 
\begin{equation}\label{eq:iterated_commutator}
[g_1,\ldots, g_i]=[g_1, [g_2, \ldots, [g_{j-1},[g_j,\ldots,g_i]]
\end{equation}
for all $1\leq j< i$.

\begin{lemma} Let $G$ be a nilpotent group with a generating set $S$. Every element in $C^j(G)$ is a product of iterated commutators $[s_1,\ldots , s_{j+1}]$, with each $s_i$ in $S$. \end{lemma}
\begin{proof}
We will proceed by induction on $j$. When $j = 1$, the statement is that every commutator $[g_1,g_2]$ is a product of commutators of the form $[s_1,s_2]$, with $s_i \in S$. This statement is true because the map $G\times G \rightarrow C^1(G)$ taking $(g_1,g_2)$ to $[g_1,g_2]$ is a group homomorphism with respect to each coordinate.  So we can write $g_1$ and $g_2$ as a product of elements in $S$, giving us that $[g_1,g_2]$ is a product of the corresponding commutators of the form $[s_1,s_2]$.

Now, suppose the statement is true for $j \leq j_0$, and consider elements in $C^{j_0+1}(G)$. By definition, each element of $C^{j_0+1}(G)$ is a product of elements of the form $[g_1,g_2]$, where $g_2 \in C^{j_0}(G)$. We can write each $g_1$ as a product of elements of $S$ and each $g_2$, by induction, as a product of terms of the form $[s_2,\ldots, s_{j_0+2}]$. Again using that the map $G\times C^{j_0+1}(G)\to G$ is a group homomorphism with respect to each coordinate, we see that each $[g_1,g_2]$, and thus each element of $C^{j_0+1}(G)$, is a product of terms of the form $[s_1, \ldots, s_{j_0+2}]$.
\end{proof}

\begin{lemma}\label{lem:commutator_with_h}
Let $G$ be an $l$-step nilpotent group, and let $S$ be a symmetric subset of $G$. Then for every $h\in G$ and every $g\in \langle S\rangle$, $[g,h]$ can be written as a product of iterated commutators $[x_1,\ldots,x_i]$, with $i\leq l$, where for each $j=1\ldots i$, $x_j\in  S\cup\{h^{\pm}\}$, and at least one of the $x_j\in \{h^{\pm}\}$.
\end{lemma}

\begin{proof}
Without loss of generality we can suppose that $G$ is the free nilpotent group of class $l$ generated by $S \cup \{h^{\pm}\}$. We shall prove the lemma by induction on $l$. For $l=1$, the statement is obvious as the group $G$ is abelian. 

Let us assume that the statement is true for $l\leq l_0$ and suppose that $G$ is $(l_0+1)$-step nilpotent. Let $\pi:G \rightarrow G/C^{l_0}(G)$ be the projection map. The group $G/C^{l_0}(G)$ is $l_0$-step nilpotent, so we can apply the induction hypothesis to $\pi([g,h]) \in G/C^{l_0}(G )$ with generating set $\pi(S\cup\{h^{\pm}\})$. This gives us that $\pi([g,h])$ is the product of elements of the form $[\pi(x_1), \ldots , \pi(x_i)]$ where $i\leq l_0$, each $x_j\in  S\cup\{h^{\pm}\}$, and at least one $x_j$ in each term is in $\{h^\pm\}$. 

Since $\pi$ is a homomorphism, $[\pi(x_1),\ldots ,\pi(x_i)] = \pi([x_1,\ldots,x_i]),$ so there is an element $z \in G$ which is a product of elements of the form $[x_1,\ldots,x_i]$, where $i\leq l_0$, each $x_j\in  S\cup\{h^{\pm}\}$, and at least one $x_j$ in each term is in $\{h^\pm\}$, such that $\pi([g,h]) = \pi(z)$. Thus, $[g,h] = yz$ for some $y \in C^{l_0}(G)$. By the previous lemma, we can write $y$ as a product of terms of the form $[a_1, \ldots,a_{l_0+1}]$ with $a_i\in S\cup\{h^{\pm}\}$.

We have shown that $[g,h]$ can be written as a product of iterated commutators of elements in $S \cup \{h^{\pm}\}$; it remains to show that each term from $y$ contains at least one $x_j \in \{h^{\pm}\}$. The terms obtained from $y$ commute with each other, so we can gather the terms without $h^{\pm}$ into a single word $w$. Let $N$ denote the normal subgroup generated by $h$. Since $[g,h] = h^g h^{-1}$, we know that $[g,h] \in N$. Similarly, each iterated commutator containing $h^{\pm}$ is in $N$, so $w \in N$. But $w$ is also in the subgroup $H$ generated by $S$, and because $G$ is the free nilpotent group of class $(l_0+1)$ generated by $S \cup \{h^\pm\}$, we have that $N \cap H$ is trivial. Thus $w$ is trivial. This completes the proof of the lemma.
\end{proof}

\begin{corollary}\label{cor:normalclosure}
Let $G$ be an $l$-step nilpotent group generated by some symmetric set $S$. Then for every element $h\in G$, the normal subgroup generated by $h$ is generated as a subgroup by the elements $h^x$, where $x\in (S\cup \{h^{\pm}\})^k$, with $k\leq 4^l$.  
\end{corollary}
\begin{proof}
Note that $h^g=[g,h]h$. Applying Lemma \ref{lem:commutator_with_h} to the commutator $[g,h]$ yields a product of iterated commutators with letters in $S\cup \{h^{\pm}\}$, where $h^{\pm}$ appears at least once in each. Let  $w=[s_1,\ldots,h,\ldots, s_i]$ be such an iterated commutator. 
By (\ref{eq:iterated_commutator}), we can rewrite it as $[s_1,\dots,s_{j-1},[h,[s_{j+1},\ldots,s_i]]]$.
Being an iterated commutator of length at most $l$, $u=[s_{j+1},\ldots,s_i]$ has word length at most $3^l$.
 Indeed, the length $k_l$ of such a commutator satisfies $k_2=4$, and $k_{l+1}=2k_l+2\leq 3k_l$, for $l\geq 2$. Thus, $k_l\leq 3^l$. It follows that $[h,u]$ is a product of conjugates of $h^{\pm1}$ by elements of length at most $3^l$. 

Now observe that given $v$ a product of commutators of $h^{\pm1}$ by elements of length at most some integer $m$, and $s\in S$, the commutator $[s,v]=v^{s}v^{-1}$ is a product of commutators of $h^{\pm1}$ by elements of length at most $m+1$. 
Now, letting $v=[h,u]$, so that $w=[s_1,\ldots, s_{j-1},v]$, we obtain from an immediate induction on $j$ that 
$w$ is a product of commutators of $h^{\pm1}$ by elements of length at most $3^l+l\leq 4^l$.

\end{proof}
The following lemma is of independent interest (compare \cite{PPSS}). 
\begin{lemma}\label{lem:Torsion Nilpotent}
Let $X$ be a transitive graph with degree $d<\infty$, and let $G$ be a nilpotent group acting faithfully and transitively on $X$. Let $H \subset G$ be the stabilizer of vertex $x$. Then  $H$ is finite of size in $O_{d,r,l}(1)$, where $r$ and $l$ are the rank and step of $G$. More precisely, all elements of $H$ have finite order, and for every prime $p$, $p^{n_p}\leq d^{4^{l+1}}$, where $p^{n_p}$ is the maximal order of a $p$-torsion element of $H$.
\end{lemma}

\begin{proof} 
 The set of vertices of $X$ can be identified with $G/H$, with $x$ corresponding to the trivial coset $H$. Faithfulness of the action is equivalent to the fact that $H$ does not contain any non-trivial normal 
subgroup of $G$. As recalled in \S \ref{sec:transitive}, $X$ is isomorphic to the Cayley-Abels $(G,H,S)$, with $S$ symmetric such that $S=HSH$. In particular one has $H\subset S^2$ and $S\cup H\subset S^{4}$.

We will first show that all elements are torsion, and simultaneously establish the bound on the order of $p$-torsion elements. 
The proof  goes as follows: we let $p$ be a prime, and for every element $h$ of $H$, suppose that $1,h,h^2, \ldots, h^{p^{n-1}}$ are distinct. We will show that there is a vertex of $X$ in the ball of radius $4^{l+1}$ around $x$ whose orbit under the action of the subgroup generated by $h$ has cardinality $\geq p^n$. Since $H$ preserves distance to $x$, this shows that $p^n \leq d^{4^{l+1}}$. This allows us to conclude that $p^{n_p} \leq d^{4^{l+1}}$, giving us the desired bound.

By assumption, the normal subgroup generated by $h^{p^{n-1}}$ is not contained in $H$. Thus by  Corollary \ref{cor:normalclosure}, there exists $g\in (S\cup H)^{4^l}\subset S^{4^{l+1}}$ such that $g^{-1}h^{p^{n-1}}g$ does not belong to $H$. Suppose now, for sake of contradiction, that $h^{i}gH=h^{j}gH$ for some $1\leq i<j\leq p^n$. Then $y=g^{-1}h^{i-j}g\in H$.  Write $i-j = p^ab$, with $a<n$ and $b$ coprime to $p$, and let $c$ be such that $cb = 1\mod{p^n}$. Because $H$ contains $y$, $H$ also contains $y^{c p^{n-1-a}} = g^{-1}h^{p^{n-1}}g$. But we know that $g^{-1}h^{p^{n-1}}g$ is not in $H$, giving our desired contradiction. Hence, the $h^igH$ are distinct for every $1 < i < j \leq p^n$. Since $g$ has length at most $4^{l+1}$, and the left-translation by $H$ preserves the distance to the origin, each $h^igH$ is in the ball of radius $4^{l+1}$. Then there are $p^n$ distinct elements in the ball of radius $4^{l+1}$, and so $p^n \leq d^{4^{l+1}}$, as desired.

Now let $h$ be any element of $H$. We have seen that it is torsion, so let $m$ be its order, and let $m=p_1^{n_1}\ldots p_j^{n_j}$ be the prime decomposition of $m$. We have seen that for each $1\leq i\leq j$, $p_i^{n_i}\leq d^{4^{l+1}}$. Since the $p_i$'s are distinct integers contained in the interval $[1,d^{4^{l+1}}]$, there are at most $d^{4^{l+1}}$ of them. It follows that $m\leq (d^{4^{l+1}})^{d^{4^{l+1}}}$. 

On the other hand, the rank $t$ of $H$ is in $O_{r,l}(1)$ (this is clear if $G$ is abelian, and easily follows by induction on the step of $G$). 
So let $h_1, \ldots h_{t}$ be generators of $H$, of orders respectively $m_1,\ldots, m_t$. Since $H$ is $l$-step nilpotent, the group defined by the following presentation\footnote{Recall that the group associated to a presentation $\langle a_1,\ldots,a_t|r_1,\ldots r_k \rangle$, where $r_1,\ldots r_k$ are words in the letters $a_1,\ldots,a_t$ is the quotient of the free group $\langle a_1\rangle$ by the normal subgroup generated by the $r_i$'s.} maps surjectively to $H$:
$$\langle a_1,\ldots,a_t|a_i^{m_i}, [g_1,\ldots,g_l]  \rangle,$$
where $(g_1,\ldots, g_l)$ run over all $l$-tuplets of words in $a_1^{\pm 1},\ldots a_t^{\pm}$. This group is finite and its cardinality only depends on $l$ and $r$, so we are done.
\end{proof}

\begin{corollary}\label{cor:VirtNilpStab}
Lemma \ref{lem:Torsion Nilpotent} extends to virtually nilpotent groups. Namely, let $X$ be a transitive graph with degree $d<\infty$, and let $G$ be a group acting faithfully and transitively on $X$. Assume $G$ has a $(r,l)$-nilpotent subgroup of finite index $i$. Then the cardinality of any vertex stabilizer $H$ is in $O_{d,r,l,i}(1)$.
\end{corollary}
\begin{proof}
This easily follows from Lemma \ref{lem:subgraph} and Lemma \ref{lem:Torsion Nilpotent}. 
\end{proof}

\subsection{Pansu's theorem}\label{secPrelim:pansu}
The material of this section will be used in Section \ref{section:pansu} for the proof of Theorem \ref{theorem:Main'}, i.e.\ for the fact that the limiting metric is polyhedral. The reader only interested in the proof of Theorem \ref{theorem:Main} can therefore ignore it.

Let $N$ be a nilpotent connected Lie group, and let $\n$ be its Lie algebra, and let $\m$ be a vector subspace supplementing $[\n,\n]$ equipped with a norm $\|\cdot\|$. 
A smooth path $\gamma: [0,1]\to N$ is said to be horizontal if $\gamma(t)^{-1}\cdot \gamma'(t)$ belongs to  $\m$ for all $t\in [0,1]$. The length of $\gamma$
with respect to $\|\cdot \|$ is then defined as $$l(\gamma)=\int_0^1\|\gamma(t)^{-1}\cdot \gamma'(t)\|dt.$$
It can be shown that since $\m$ generates the Lie algebra $\n$, every pair of points can be joined by a horizontal path (see \cite{Gro}). 
The Carnot-Caratheodory metric associated to $\|\cdot\|$ is defined so that the distance between two points in $N$ is 
$$d(x,y)=\inf_{\gamma} \{l(\gamma); \; \gamma(0)=x,\; \gamma(1)=y\},$$
where the infimum is taken over all piecewise horizontal paths $\gamma$. Note that if $N=\R^m$, so that $\m=\n$, then the Carnot-Caratheodory metric is just the distance associated to the norm $\|\cdot\|$, and more generally if $N$ is abelian, $d$ is the Finsler metric obtained by quotienting the normed vector space $(\R^m,\|\cdot\|)$ by some discrete subgroup. Finally we shall need the following simple remark.
Suppose $p:N\to N'$ is a surjective homomorphism between nilpotent connected Lie groups, and let $d_{cc}$ be a Carnot-Caratheodory metric on $N$ associated to $(\m,\|\cdot\|)$. Then there exists a unique metric $d'$ on $N'$ defined so that $$d'(x,y):=\inf_{x,y} \{d_{cc}(x,y); \; p(x)=x', \; p(y)=y'\}.$$ 
This metric is easily seen to be the Carnot-Caratheodory Finsler metric on $N'$ associated to  the normed vector space $(\m',\|\cdot\|')$ obtained by projecting  $(\m,\|\cdot\|)$. In particular if $N'=N/[N,N]$, then $d'$ is simply given by the norm $\|\cdot\|.$ 

\begin{thmm}\label{thmPrelim:Pansu} \cite{Pansu}
Let $N$ be a finitely generated virtually nilpotent group equipped with some finite generating set $T$. Then $(N,d_T/n,e)$ converges to some simply connected nilpotent Lie group $N_{\R}$ equipped with some Carnot-Caratheodory metric $d_{cc}$.
\end{thmm}
Our goal now is to show that $d_{cc}$ is associated to some polyhedral norm.
Now let $N_0$ be a nilpotent normal subgroup of finite index in $N$ and let $\pi$ be the projection modulo $[N_0,N_0]$. We have that $(N/[N_0,N_0],d_{\pi(T)}/n)$ converges to $(\R^m,\|\cdot\|)$. Therefore it is enough to consider the case where $N=A$ is virtually abelian.

\begin{lemma}\label{lemPrelim:polyhedral}
Let $A$ be a virtually abelian group, and let $U$ be a finite generating subset. Then the  $(A,d_U/n,e)$ converges to a polyhedral normed space.   
\end{lemma}

Together with the previous discussion, this yields 
\begin{corollary}\label{corPrelim:Polyhedral}
The Carnot-Caratheodory metric on $N_{\R}$ in Pansu's theorem is associated to a polyhedral norm. 
\end{corollary}

Before proving Lemma \ref{lemPrelim:polyhedral}, let us introduce some notation and prove a preliminary fact of independent interest. Let $A$ be a virtually abelian group, and let $U$ be a finite generating subset. Let $A_0$ be a normal finite index abelian subgroup of $A$, and let $B=A/A_0$. Note that the action by conjugation of $A$ on $A_0$ induces an action of the finite group $B=A/A_0$ on $A_0$, that we will denote by $b\cdot a$. 
\begin{lemma}\label{lemPrelim:virtabelian}
With the above notation, set $V=U^{|B|}\cap A_0$ Let $W=\{b\cdot v,; \; v\in V\}$. For every $w\in W$ written as $b\cdot v$, we consider its weight $\omega(w)=|v|_U$ and let $|a|_{W,\omega}$ be the corresponding weighted word length, i.e.\ the minimum $n$ such that $a=(b_1\cdot v_1)(b_2\cdot v_2)\ldots$ with $v_i\in V$ and  $\omega(v_1)+\omega(v_2)+\ldots=n.$ Then for all $a\in A_0$, one has
$$|a|_{W,\omega}\leq |a|_U\leq |a|_{W,\omega}+2|B||U|^{|B|}.$$
\end{lemma}
\begin{proof}
Consider an element $a\in A_0$ of $U$-length $n$ and write it as a word with $n$ letters in $U$. Observe that each word of length $|B|$ must contain a non-empty subword that lies in $A_0$. If $|a|_U\geq |B|$, we therefore have $a=vv'v"$, where $v'$ is a non-empty subword corresponding to an element of $V$. We can shift $w'$ to the left up to conjugating it by $w$. We obtain $a=(ww'w^{-1})ww"=(b\cdot w')ww"$. 
We can repeat this until we get a product of terms of the form $(b_i\cdot w_i)$, where each $w_i\in V$, $b_i\in B$ and the sum of the $|w_i|_U$ is exactly $n$. Therefore, we have 
$$|a|_{W,\omega}\leq |a|_U.$$
Note that this inequality did not use the fact that $A_0$ is abelian.

Conversely, let $a$ be such that $|a|_{W,\omega}=n$. Then $a$ is a product of product of terms of the form $(b_i\cdot w_i)$, where each $w_i\in V$, $b_i\in B$ and the sum of the $|w_i|_U$ is exactly $n$.
Up to permuting these (commuting) terms, we can gather them as 
$$a= \Pi_{b\in B,v\in V}b\cdot a_{v,b}$$
Since for each $b\in B$ we can choose $g\in V$ such that $b\cdot a_{v,b}=ga_{v,b}g^{-1}$, the other inequality follows.
\end{proof}

\noindent{\bf Proof of Lemma \ref{lemPrelim:polyhedral}.}
Let us first prove the lemma when $A\simeq \Z^m$. Let $U=(u_1^{\pm 1}, \ldots, u_k^{\pm 1})$, and consider the morphism $\Z^k\to \Z^m$ mapping the standard generators $e_1,\ldots, e_k$  to $u_1,\ldots, u_k$. The fact that $\Z^k$ with its $\ell^1$-metric divided by $n$ converges to $\R^k$ with the $\ell^1$-norm is a trivial observation. It follows that the unit ball of the limit of $(A,d_U/n)$ is simply the projection of the $\ell^1$-norm of $\R^k$, so it is polyhedral.

The general case reduces to the abelian case, with weighted generators thanks to Lemma \ref{lemPrelim:virtabelian}. \hfill $\square$

\section{Proofs of Theorems \ref{theorem:Main} and  \ref{theorem:Main'}}
\label{section:MainTheorem}

\subsection{Main step: reduction to virtually nilpotent groups}
\label{section:nilpotent}

The first (and main) step of the proofs of both Theorem \ref{theorem:Main} and Theorem \ref{theorem:Main'} is the following reduction. Given a group $G$ with finite generating set $S$, we will let $(G,S)$ denote the corresponding Cayley graph. Let $D_n$ denote the diameter of $X_n$. A sequence of groups is said to be uniformly virtually nilpotent if these groups have nilpotent subgroups whose index is bounded above by some constant.

\begin{proposition}\label{prop:NilpotentReduction} Let $(X_n)$ be a sequence of vertex transitive graphs, and let $G_n$ be a discrete subgroup of the isometry group of $X_n$, acting transitively. Let $D_n\leq \Delta_n$ such that $D_n \rightarrow \infty$ such that  $|B(x,D_n)| = O({D_n}^q)$. Then there exists a sequence of uniformly virtually nilpotent groups $N_n$ with uniformly bounded step and generating sets $T_n$ such that $X_n$ is $(1,o(\Delta_n))$-quasi-isometric to  $(N_n,T_n)$. Moreover, if we assume in addition that the degree of $X_n$ is  bounded independently of $n$, then we can take $T_n$ with bounded cardinality, and $|N_n|$  to be less than a constant times $|X_n|$. 
\end{proposition}

\begin{remark}\label{rem:nilreduction} 
In the proof of the proposition, $N_n$ appears as a group of isometries, though not of $X_n$, but rather of a quotient graph $Y_n$. Precisely we have the following diagram:
\begin{equation}
  \xymatrix{  &
      X_n \ar[d]\\
            (N_n,T_n)  \ar[r] &
     Y_n}
\end{equation}
where the arrows are graph projections whose fibers have diameter in $o(\Delta_n)$.   
\end{remark}

We will need the following lemmas.

\begin{lemma}[Doubling in $X_n$] \label{doubling_in_X} Under the assumption of the previous proposition, there exists a $K$ depending only on $q$ so that the following holds. There exists a sequence $R_n\to \infty$, with $R_n = o(D_n)$ such that $|B_{X_n}(100R_n)| \leq K|B_{X_n}(R_n)|$.  \end{lemma}
\begin{proof} Suppose by contradiction that $|B_{X_n}(100R)| > K|B_{X_n}(R)|$ for all $R_0<R<D_n^{1/2}$. Then $$|X_n| \geq |B_{X_n}(D_n^{1/2})| > K^{\log_{100}{(D_n^{1/2}/R_0)}} |B_{X_n}(R_0)| \geq CD_n^{(1/2)\log_{100}{K}},$$ for some $C$ independent of $n$. Letting $K=100^{2q}$, there is some $N$ such that for $n>N$ this cannot hold.  \end{proof}

Fix some $x \in X_n$. Let $H_n$ be the stabilizer of $x$ in $G_n$, and $S_n$ be the subset of $G_n$ containing all $g$ such that $g(x)$ is a neighbor of $x$. When it is clear from context, let $G_n$ denote the Cayley graph of $G_n$ with generating set $S_n$.

\begin{lemma}[Doubling in $G_n$]\label{doubling_in_G}  There exists $K>0$ depending only on $q$ so that the following holds. There exists a sequence $R_n\to \infty$, with $R_n = o(D_n)$ such that $|B_{G_n}(100R_n)| \leq K|B_{G_n}(R_n)|$. \end{lemma}
\begin{proof}
Recall from~\S\ref{sec:transitive} that the graph homomorphism from the Cayley graph $(G_n, S_n)$ to the Cayley-Abels graph $X_n=(G_n,H_n, S_n)$ sends $B_{G_n}(r)$ to $B_{X_n}(r)$ for all $r\geq 0$. Hence
we have 
$$|B_{G_n}(r)|=|H||B_{X_n}(r)|.$$
In particular, the doubling condition for $X$ is equivalent to the doubling condition for $(G,S)$, with the same constant. \end{proof}

The main tool in our proof is the following theorem from \cite{BGT} (although not exactly stated this way in \cite{BGT}, it can be easily deduced from \cite[Theorem 1.6]{BGT} using the arguments of the proof of \cite[Theorem 1.3]{BGT}).

\begin{theorem}[BGT] \label{BGT} Let $K \geq 1$. There is some $n_0\in \N$, depending on $K$, such that the following holds. Assume $G$ is a group generated by a finite symmetric set $S$ containing the identity. Let $A$ be a finite subset of $G$ such that $|A^5| \leq K|A|$ and $S^{n_0} \subset A$. Then there is a finite normal subgroup $F \triangleleft G$ and a subgroup $G' \subset G$ containing $F$ such that
\begin{itemize}
\item $G'$ has index $O_K(1)$ in $G$
\item $N = G'/F$ is nilpotent with step and rank $O_K(1)$.
\item $F$ is contained in $A^{O_K(1)}.$
\end{itemize}
\end{theorem}

\begin{proof}[Proof of Proposition \ref{prop:NilpotentReduction}]
By Lemma~\ref{doubling_in_G}, we can find a sequence $R_n$ with both $R_n$ and $\Delta_n/R_n$ tending to infinity such that $|B_{G_n}(100R_n)|\leq K|B_{G_n}(R_n)|$. Then by Theorem~\ref{BGT} applied to $A=B_{G_n}(R_n)$, we obtain a sequence of groups $F_n$ and $M_n = G_n/F_n$ such that $F_n$ has diameter $o(\Delta_n)$, and $M_n$ is uniformly virtually nilpotent with uniformly bounded step and rank. 

Proposition \ref{prop:NilpotentReduction} now results from the following facts.
\begin{itemize}

\item  Let $Y_n$ be the graph obtained by dividing $X_n$ by the normal subgroup $F_n$. The quotient map $X_n\to Y_n$ has fibers of diameter $o(\Delta_n)$, so $Y_n$ is $(1,o(\Delta_n)$-QI equivalent to $X_n$. Note that $M_n$ acts transitively on $Y_n$, and that if $X_n$ had bounded degree, then $Y_n$ also has bounded degree.

\item Let $L_n$ be the kernel of the action of $M_n$ on $Y_n$, and define $N_n=M_n/L_n$. Then $N_n$ acts faithfully and transitively on $Y_n$, so $Y_n$ is isomorphic to $(N_n,H_n, T_n)$, where $H_n$ is the stabilizer of a vertex $x$, and $T_n$ is the set of elements of $N_n$ taking $x$ to a neighbor of $x$. As in the proof of Lemma~\ref{doubling_in_G}, $Y_n$ is $(1, O(1))$ quasi-isometric to $(N_n, T_n)$.

\item (For the second statement of Proposition  \ref{prop:NilpotentReduction}) Assume here that the degree of $X_n$ is bounded. Recall that $N_n$ is uniformly virtually nilpotent of bounded step and rank and $Y_n$ has bounded degree. Thus, we deduce from Lemma \ref{lem:Torsion Nilpotent} and Corollary~\ref{cor:VirtNilpStab} that $H_n$ has uniformly bounded cardinality, and so does $T_n$.

\end{itemize}
Recapitulating, we constructed quasi-isometries between $X_n$ and $Y_n$, and finally between $Y_n$ and the Cayley graph of $N_n$. The multiplicative constants of these quasi-isometries are all equal to $1$, and the additive ones are bounded by $o(\Delta_n)$, so the proposition is proved.
\end{proof}

\subsection{The limit in Theorems \ref{theorem:Main} and \ref{theorem:Main'} is a torus with a Finsler metric}\label{section:Main}

We will use the following result, based on \cite{BG} and \cite{BGT}, which was communicated to us by Breuillard (he informed us that he will include a proof of this fact in a subsequent note with his co-authors Green and Tao).

\begin{lemma}\label{lem:doubling implies doubling}
Let $G$ be a group generated by some symmetric finite subset $S$. For every $K>0$, there exists $n_0\in \N$ and $K'>0$ such that if for some $n\geq n_0$, $|S^{100n}|\leq K|S^n|$, then for all $m\geq n$, one has $|S^{100m}|\leq K'|S^m|$.
\end{lemma}

Here is a generalization of Theorem \ref{theorem:Main}.

\begin{theorem}\label{theorem:MainInfinite}
let $D_n\leq \Delta_n$ be two sequences going to infinity, and let $(X_n)$ be a sequence of finite vertex transitive graphs such that the balls of radius $D_n$ satisfy $|B(x,D_n)|=O(D_n^q)$. We also suppose\footnote{The discreteness of $G_n$ was used in Proposition \ref{prop:NilpotentReduction}. We will come back to this 
unnatural assumption in \S \ref{section:discussionInfinite}.} that the isometry group of $X_n$ has a discrete subgroup $G_n$  acting transitively. Then $(X_n,d/\Delta_n)$ has a subsequence converging for the (pointed) GH-topology to a connected nilpotent Lie group equipped with a left-invariant Carnot-Caratheodory metric.   
\end{theorem}

The fact that the scaling limit in Theorem~\ref{theorem:Main} is a torus equipped with a Finsler metric follows from Theorem \ref{thmPrelim:PW}.

Recall that in the proof of Lemma \ref{doubling_in_G}, we showed that the group $G_n$, equipped with some suitable word metric, is such that the volume of its balls of a given radius equals a constant times the volume of balls of same radius in $X_n$.  Therefore, doubling at a single scale $R_n=o(D_n)$ in $G_n$ (established in Lemma~\ref{doubling_in_G}) implies doubling at any scale $\geq R_n$ in $G_n$ by Lemma \ref{lem:doubling implies doubling}, which then implies doubling at any scale $\geq R_n$ in $X_n$. Thus, the $X_n$ are relatively compact for the Gromov-Hausdorff topology (see for instance \cite[Proposition 5.2]{Gromov}).  Let $X$ be an accumulation point: this locally compact space satisfies a doubling condition at any scale, which is easily seen to imply finite Hausdorff dimension. Observe that the Gromov-Hausdorff distance from $X_n$ to its $1$-skeleton is $1$, and that the latter is a geodesic metric space.  Hence because the scaling limit of $X_n$ is $X$, its rescaled 1-skeleton also converges to $X$, which is therefore a geodesic locally compact finite dimensional metric space.

By Proposition  \ref{prop:NilpotentReduction}, we know that there are virtually nilpotent groups $N_n$ with generating sets $T_n$ such that the Cayley graphs $(N_n, T_n)$ are $(1, o(\Delta_n))$-quasi-isometric to $X_n$. Thus, $X$ is isometric to the scaling limit of $(N_n, T_n)$.  The latter obviously admits a virtually nilpotent group of isometries (with index and step bounded independently of $n$), hence by the discussion at the end of \S \ref{section:HG}, we know that the isometry group of $X$ has a closed, virtually nilpotent subgroup acting transitively. Finally, Theorem \ref{thmPrelim:PW} implies that $X$ is a connected nilpotent Lie group equipped with an invariant Carnot-Caratheodory metric. 

\subsection{Reduction to Cayley graphs of abelian groups}\label{section:abelian}
From now on, we come back to the setting and notation of Theorem \ref{theorem:Main}.
In Section \ref{section:Main}, we showed that $X_n$ is relatively compact for the Gromov-Hausdorff distance and that accumulation points are finite dimensional tori. Moreover, we proved along the way that there exists a sequence of finite Cayley graphs $(N_n,T_n)$ with the same accumulation points as $X_n$,  such that $N_n$ is virtually nilpotent with bounded index and step. We let $\hat{d}_n$ be the bi-invariant metric introduced in \S \ref{section:HG} on $N_n$ seen as a group of isometries of $(N_n, T_n)$.

Let $N_n^0$ be a nilpotent normal subgroup of bounded index in $N_n$.  Let $A_n = N_n/[N_n^0, N_n^0]$, and let  $A_n^0$ be the abelianization of $N_n^0$. Let $d_n$ be the distance on $N_n^0$ induced by the word metric on $(N_n, T_n)$, divided by the diameter $D_n$ of $(N_n,T_n)$, and let $d_n'$ be the distance on $A_n^0$ induced by the word metric on $(A_n,\pi_n(T_n))$ divided by $D_n$, where $\pi_n:N_n \rightarrow A_n$ is the projection map. 

\begin{proposition}\label{prop:reductionabelian}
Suppose $(N_n,d_n)$ converges to some torus $(X,d)$. 
Then the sequence of Cayley graphs $(A_n,\pi_n(T_n))$ whose distance is rescaled by $D_n$, as well as  $(A_n^0,d'_n)$ have subsequences converging to $(X,d)$.
\end{proposition}

\begin{proof}
Being at bounded Gromov-Hausdorff distance, these two sequences have the same limit. Let us focus then on $(A_n^0,d_n')$.
Since it has bounded index in $N_n$, $(N_n^0,d_n)$ converges to $(X,d)$. The projection from $(N_n^0, d_n)$ to $(A_n^0, d'_n)$ is 1-Lipschitz, so it suffices to show that the fibers have diameter in $o(1)$. By transitivity,  this is equivalent to checking that any sequence $g_n\in [N_n^0,N_n^0]$ satisfies $d(1,g_n)\to 0$.  

Let $\hat{d}_n^0$ denote the metric on $N_n^0$ induced by $\hat{d}_n$. By Corollary \ref{corPrelim}, the ultralimit $N^{0}=\lim_{\omega}(N_n^0, \hat{d}_n^0)$ identifies to a compact nilpotent transitive group of isometries of $X$. Therefore $N^{0}$ must be abelian. 

 By Lemma \ref{lem:CommutatorLength} (below), $g_n$ can be written as a bounded product of commutators, each of which converges to a commutator in $N^0$.
It follows that $\hat{d}_n^0(1,g_n)\to 0$ for any sequence $g_n\in [N_n^0,N_n^0]$. The proposition then follows from the inequality $d_n\leq\hat{d}_n^0$. 
\end{proof}

\begin{lemma}\label{lem:CommutatorLength}
Let $N$ be some $l$-step nilpotent group. Then every element $x$ in $[N,N]$ can be written as a product of $l$ commutators.
\end{lemma}
\begin{proof}
The statement is easy to prove by induction on $l$. There is nothing to prove if $l=1$, so let us assume that $l>1$. Note that $C^{l}(N)=\{1\}$, so that $C^{l-1}(N)$ is central. Let $x\in [N,N]$. By induction, $x$ can be written as a product of $l-1$ commutators times an element of  $C^{l-1}(N)$. Hence it is enough to prove that every element of $C^{l-1}(N)$ can be written as a single commutator.
Recall that the iterated commutator $[x_1,[x_2[\ldots ,x_l]\ldots]$ induces a morphisms from $\bigotimes_{i=1}^l A$ to $C^{l-1}(N)$, where $A=N/[N,N].$ In particular its range is a subgroup of $C^{l-1}(N)$. Since it contains generators of $C^{l-1}(N)$ it is equal to $C^{l-1}(N)$, which proves the lemma.
\end{proof}

To complete the reduction, we want to find abelian groups $(B_n, V_n)$ that have diameter of the order of $D_n$ and volume at most that of $X_n$, that converge to $(X,d')$ where $d'$ is bilipschitz to $d$, and such that if the degree of $X_n$ is bounded, so is $V_n$. 

To do this, let $Y_n$ be as in Section~\ref{section:nilpotent}. Then we can extend the diagram from Remark~\ref{rem:nilreduction} as follows:
\begin{equation}
  \xymatrix{  &
      X_n \ar[d]\\
      		(N_n, T_n)\ar[d] \ar[r] & Y_n \ar[d]\\
            	(A_n,U_n) \ar[r] & Z_n\\
		(A_n^0, U_n^0) \ar[u]\ar[r] & Z_n' \ar[u] 
	}
\end{equation}
where $Z_n$ is the quotient of the Caylay-Abels graph $Y_n$ of Remark \ref{rem:nilreduction} by the normal subgroup $[N_n^0,N_n^0]$, $Z_n'$ is the graph guaranteed in Lemma~\ref{lem:subgraph}, and the bottom vertical arrows are (O(1),O(1))-QI inclusion as in Lemma \ref{lem:subgraph}. 

If the degree of $X_n$ is bounded then the degree of $Z_n$ is bounded. By Lemma \ref{lem:subgraph}, if the degree of $Z_n$ is bounded then the degree of $Z_n'$ is also bounded. Moreover, $|Z_n'| \leq |Z_n| \leq |Y_n| \leq |X_n|$.

The graph $Z_n$  can be described as the Cayley-Abels graph of $(A_n,K_n,U_n)$ for some subgroup $K_n$ with respect to the bi-$K_n$-invariant symmetric subset $U_n=K_nU_nK_n$, so that the horizontal arrows from $(A_n, U_n)$ to $Z_n$ and, similarly, from $(A_n^0, U_n^0)$ to $Z'_n$, are graph projections.  By the previous proposition the fibers of these projections have diameter in $o(D_n)$, allowing us to conclude both that $Z_n$ has diameter on the order of $D_n$ and that $Z'_n$ has scaling limit $(X,d')$, where $d'$ is bilipschitz to $d$.

Since $A_n^0$ is abelian, the stabilizers of its action on $Z'_n$ are simply its kernel, so that $Z_n'$ itself is the Cayley graph of some abelian group, denoted $(B_n,V_n)$. This concludes our reduction to the case where $X_n$ is the Cayley graph of some abelian group.




\subsection{Bound on the dimension of the limiting torus}\label{section:dimension}

In this section we prove the bound on the dimension of the limiting torus in Theorem \ref{theorem:Main'}. By Propositions \ref {prop:NilpotentReduction} and \ref{prop:reductionabelian}, we can assume that $X_n$ is a sequence of Cayley graphs of abelian groups associated to generating sets of bounded --hence up to a subsequence, fixed-- cardinality. 

We shall use repeatedly the following easy fact.

\begin{lemma}\label{lem:subconvergenceAbelian}
Let $(G_n,S_n)$ be a sequence of Cayley graphs of abelian groups such that $S_n=\{\pm e_1(n),\ldots, \pm e_k(n)\}$ for some fixed $k\in \N$. Then for any sequence $\Delta_n$ going to infinity, the rescaled sequence $(G_n,d_{S_n}/\Delta_n)$ subconverges to some connected abelian Lie group of dimension at most $k$.  
\end{lemma}
\begin{proof}
Let $S$ be the standard generating set of $\Z^k$. The Cayley graph $(\Z^k,S)$ has a $\Delta_n$-rescaled limit equal to $\R^k$ (equipped with the standard $\ell^1$-metric). The lemma follows from this fact by considering the projection $\Z^k\to G_n$ mapping $S$ to $S_n$. 
\end{proof}

 The bound on the dimension in Theorem \ref{theorem:Main'} immediately results from the following lemma.  
\begin{lemma}\label{lem:dim_bound1}
Let $(B_n,V_n)$ be a sequence of finite Cayley graphs where $B_n$ is abelian and $V_n=\{\pm e_1,\ldots, \pm e_k\}$ for some fixed $k$. Let $\Delta_n$ be a sequence going to $+\infty$. Then up to taking a subsequence, the rescaled sequence $(B_n,V_n/\Delta_n)$ converges to a connected abelian Lie group $X$ of dimension at most the largest integer $j$ such that $\Delta_n^{j}=O(|B_n|)$.
\end{lemma}

Before proving the lemma, let us introduce a useful definition.

\begin{definition}
Let $(G,S)$ be a Cayley graph where $G$ is abelian, $S=\{\pm e_1,\ldots, \pm e_k\}$. Define its radius of freedom $r_f(G,U)$ to be the largest $r$ such that the natural projection $\Z^k\to G$ is isometric in restriction to the ball of radius $r$.  
\end{definition}

\begin{proof}[Proof of Lemma \ref{lem:dim_bound1}:]
We know from Lemma \ref{lem:subconvergenceAbelian} that the subsequential limit $X$ exists and is a connected abelian Lie group, and that its dimension $d$ is at most $k$. We want to show $d \leq j$. We will argue by induction on $k$. The case $k=0$ being clear, we assume that $k>0$.

If  there exists $c>0$ such that $r_f(B_n,V_n)\geq c\Delta_n$ along a subsequence, then the cardinality of $B_n$ is at least $(c\Delta_n/k)^k$. By the definition of $j$, this implies that $j \geq k$, and since $d \leq k$, we can conclude that $d \leq j$, as desired. Suppose therefore that $r_f(B_n,V_n)=o(\Delta_n)$.

\begin{itemize}
\item Claim 1: Let $r_n = r_f(B_n,V_n)$. By definition, $r_n+1$ is the smallest integer such that there exist $(n_1, \ldots , n_k)$ with $\sum |n_i| \leq 2(r_n+1)$ such that $n_1e_1 + \cdots + n_ke_k = 0$, and such that at least one of the $n_i$'s is non-zero. Up to permuting the generators, we can assume that $|n_k|\geq \sum |n_i|/k$.

\item Claim 2: Let $C_n$ be the subgroup of $B_n$ generated by the set $W_n=\{\pm e_1,\ldots, \pm e_{k-1}\}$.
The map $(C_n,W_n)\to (B_n,V_n)$ is a graph morphism and therefore is $1$-Lipschitz. Moreover it follows from the first claim that for all $r>0,$ any point in the ball of radius $r$ in $B_n$ lies at distance at most $2(r_n+1)$ from the image of the ball of radius $(k+1)r$. Indeed, let $x=m_1e_1+\ldots +m_ke_k$, with $\sum |m_i|=r$, and let $q$ be such that $m_k=qn_k+t$, with $|t|<|n_k|\leq 2(r_n+1)$ and $0\leq q\leq |m_k|/|n_k|$. Now, $x$ is at distance at most $|t|$ from $y=(m_1+qn_1)e_1+\ldots +(m_{k-1}+qn_{k-1})e_{k-1}$, which lies in the ball of radius $r'=q(\sum |n_i|)+r$. On the other hand, note that $$q\leq \frac{|m_k|}{|n_k|}\leq \frac{k|m_k|}{\sum |n_i|}\leq \frac{kr}{\sum |n_i|}.$$
Hence, $r'\leq (k+1)r$. 

\item Claim 3: If $r_n=o(\Delta_n)$, then by Lemma \ref{lem:subconvergenceAbelian} the sequence of $(C_n,W_n)$ rescaled by $\Delta_n$ subconverges to some connected abelian Lie group $Y$ which admits a continuous morphism to $X$. Note that according to the second claim, the image of $C_n$ is $o(\Delta_n)$-dense in $B_n$. This merely implies that the image of $Y$ in $X$ is dense (which is not enough for our purpose).  However, the second claim implies something stronger: namely that the image of a ball of radius $(k+1)r$ in $C_n$ is $o(\Delta_n)$-dense in the ball of radius $r$ of $B_n$ for all $r$. Passing to the limit, this implies that for all $\delta>0$, the image of a closed ball of radius $(k+1)\delta>0$ in $Y$, which is compact, is dense in --and therefore contains-- the ball of radius $\delta$ of $X$. In particular, $\dim Y\geq \dim X$.
\end{itemize}

Applying the induction hypothesis to $(C_n,W_n)$, we deduce that $\dim X\leq \dim Y\leq j$, hence the lemma follows. 
\end{proof}

\subsection{The limit in Theorem \ref{theorem:Main'} is Finsler polyhedral}\label{section:pansu}
 In this section, we conclude the proof of Theorem \ref{theorem:Main'} by showing that the limiting metric is Finsler polyhedral. We will actually prove the more general

\begin{theorem}\label{theorem:MainInfinite'}
let $C>0$,  and let $D_n\leq \Delta_n$ be two sequences going to infinity, and let $(X_n)$ be a sequence of finite vertex transitive graphs with degree $\leq C$ such that the balls of radius $D_n$ satisfy $|B(x,D_n)|=O(D_n^q)$. We also suppose that the isometry group of $X_n$ has a discrete subgroup $G_n$  acting transitively. Then $(X_n,d/\Delta_n)$ has a subsequence converging for the (pointed) GH-topology to a connected nilpotent Lie group equipped with a left-invariant polyhedral Carnot-Caratheodory metric.   
\end{theorem}

Theorem \ref{theorem:MainInfinite'} is based on the reduction to nilpotent groups (relying on the results of \cite{BGT}) and on  Pansu's theorem (\S \ref{secPrelim:pansu}). But contrary to Theorem \ref{theorem:MainInfinite}, we will not appeal to any general fact about locally compact groups. 

Thanks to Proposition  \ref{prop:NilpotentReduction}, the theorem is an immediate consequence of the following
\begin{proposition} \label{prop:Torus}
Suppose $(N_n)$ is a sequence of virtually nilpotent groups with bounded  index and step, and bounded generating sets $T_n$. Then for every $\Delta_n\to \infty$, a subsequence of the Cayley graphs $(N_n, d_{T_n}/\Delta_n,e)$ converges for the pointed Gromov-Hausdorff topology to a connected nilpotent Lie group, equipped with some polyhedral Carnot-Caratheodory metric.
\end{proposition}

Note that in Proposition~\ref{prop:Torus}, we do not require $D_n$ to be the diameter of $N_n$, but rather any unbounded sequence. Specialized in the case where the groups $N_n$ are finite, and $D_n$ is the diameter of the Cayley graph $(N_n, T_n)$, this proposition implies that the limit is a finite dimensional torus by Lemma \ref{lemPrelim:torus}, so that the metric on the limit is Finsler.

\begin{proof}[Proof of Proposition~\ref{prop:Torus}]
First, up to some subsequence we can assume that the sequence of groups $(N_n,T_n)$
converges for the topology of marked groups to some $(N,T)$ (see \cite{Cham} for the definition and the fact that the space of marked groups is compact). Observe that since $N_n$ is uniformly virtually nilpotent of bounded step, then so is the limit $N$. Note that since $N$ is finitely presented, a neighborhood of $N$ in the space of marked groups is contained the set of quotients of $N$ \cite[Lemme 2.2.]{Cham}. 
This implies that for almost all $n$, $(N_n,T_n)$ is a quotient  of $(N,T)$. We shall ignore the finitely many $n$ for which this might fail and denote by $p_n:N\to N_n$ a surjective homomorphism mapping $T$ onto $T_n$. In particular this implies $(N_n,T_n)$ has a uniform doubling constant. Hence the rescaled sequence $(N_n,d_{T_n}/\Delta_n,e)$ is GH-relatively compact. Therefore up to passing to a subsequence, we can suppose that the rescaled sequence converges to some homogeneous limit space $(X,d)$.

By Theorem \ref{thmPrelim:Pansu} and Corollary \ref{corPrelim:Polyhedral}, the sequence $(N,d_T/\Delta_n,e)$ GH-converges to some simply-connected nilpotent Lie group $(N_{\R},d_{cc})$, equipped with some Carnot-Caratheodory metric associated to a polyhedral norm.
Recall that if a sequence of (pointed) homogeneous metric spaces $(Y_n,d_n,o_n)$ GH-converges to some locally compact space $(Y,d)$, then for any ultra-filter on $\N$, and any sequence of points $o_n\in Y_n$, the corresponding ultra-limit of pointed metric spaces $(Y_n,d_n,o_n)$ is naturally isometric to $(Y,d)$ (see Lemma \ref{lemPrelim:ultra}). Given a sequence $y_n \in Y_n$, let $[y_n]$ denote the equivalence class of $(y_n)$ in the ultralimit.

Let $\tilde{p}_n:(N,d_T/\Delta_n) \rightarrow (N_n,d_{T_n}/\Delta_n)$ be the projection induced from $p_n$. The maps $\tilde{p}_n$ are 1-Lipschitz and surjective, so there is a projection $p:(N_{\R}, d') \rightarrow (X,d)$ from the limit $(N_{\R},d')$ of the sequence $(N,d_T/\Delta_n)$ to $(X,d)$ such that for each sequence $(x_n)$  in $N$,
$$[\tilde{p}_n(x_n)] = p([x_n]).$$
We also have that for every $g \in N_n$, there is a $x \in N$ so that $p_n(x) = g$ and $|g|_{T_n} = |x|_T$.
We claim that  $X$ is naturally a group. More precisely if $dist(g_n, g_n') \rightarrow 0$ and $dist(h_n, h'_n) \rightarrow 0$ then $[g_nh_n^{-1}] = [g'_nh'^{-1}_n]$. We can write $g_n' = g_na_n$ and $h_n' = h_nb_n$ with $|a_n|, |b_n| \rightarrow 0$. We can choose $x_n, c_n, y_n, d_n \in (N,d_T/D_n)$ that are mapped by $p_n$ to $g_n, a_n, h_n,$ and $b_n$, respectively, and so that $|c_n|, |d_n|\rightarrow 0$ in $(N, d_T/D_n)$. Because $(N, d_T/D_n)$ is a group, we have $[g_nh_n^{-1}] = [p(x_n)p(y_n)^{-1}] = [p(x_ny_n^{-1})] = p[x_ny^{-1}_n] = p([x_nc_nd^{-1}_ny^{-1}_n]) = [p(x_nc_nd^{-1}_ny^{-1}_n)] = [g_na_nb^{-1}_nh^{-1}_n] = [g'_nh'^{-1}_n]$. 
We deduce that $(X,d)$ is a quotient of $(N_{\R},d_{cc})$, equipped with a Carnot-Caratheodory metric associated to a norm that is a projection of the norm associated to $d_{cc}$ (see the remark before  \ref{thmPrelim:Pansu}). 
\end{proof}

\section{Proof of Theorems \ref{theorem.bis} and \ref{theorem.main}}
\label{section.bis}

In this section, we will provide elementary proofs  of Theorems~\ref{theorem.bis} and \ref{theorem.main}.

\subsection{Proof of Theorem \ref{theorem.bis}}
The proof of Theorem \ref{theorem.bis} goes approximately as follows. We show using rough transitivity that if the graph does not converge to a circle, then it must contain a caret of size proportional to the diameter. Then, iterating using rough transitivity, we generate large volume, contradicting the assumption.

First, we will define a few key terms.
\begin{definition}
Given $K\geq 0$ and $C\geq 1$, a {\em $(C,K)$-quasi-geodesic segment} in a metric space $X$ is a $(C,K)$-quasi isometrically embedded copy of the interval $[1,k]$ into $X$; i.e.\ a sequence of points $x_1,\ldots x_k\in X$ such that
$$C^{-1}(j-i)-K\leq d(x_i,x_j)\leq C(j-i)+K$$ for all $1\leq i<j\leq k$.
\end{definition}
Let $B(v,r)$ denote the ball of radius $r$ around a vertex $v$. A quasi-caret  of radius $\geq R$ consists of three quasi-geodesic segments  $\gamma_1$, $\gamma_2$ and $\gamma_3$ that start from a point $v_0$, escape from $B(v_0, R)$, and move away from one another at ``linear speed." In other words,
\begin{definition}
A {\em quasi-caret  of radius $R$} is a triple $\gamma_1$, $\gamma_2$, $\gamma_3$ of quasi-geodesics from a vertex $v_0$ to vertices $v_1$, $v_2$, and $v_3$, respectively, such that $d(v_0, v_i) = R$ for $i = 1,2,3$, and there is a constant $c>0$ satisfying that  for all $k_1,k_2,k_3$, $d(\gamma_i(k_i),\gamma_j(k_j))\geq c\max\{k_i,k_j\}$ for  $i\neq j$.
\end{definition}
If a roughly transitive graph has a quasi-caret of radius $\geq \epsilon D$, then  by moving around this caret with quasi-isometries (with uniform constants), we obtain at every point of the graph a quasi-caret  (with uniform constants) of radius $\geq \epsilon' D$.

The proof of Theorem~\ref{theorem.bis} follows from the following three lemmas.

\begin{lemma}
\label{no_small_carets_bis}
Let $D = \diam(X)$. Suppose there exists a quasi-caret of radius $R = \epsilon D$ for some $\epsilon>0$ in a finite (C,K)-roughly transitive graph $X$. Then $|X| \geq \epsilon' D^{\delta}$, where $\epsilon'$ and $\delta>1$ depend only on $C,K$ and $\epsilon$.
\end{lemma}
\begin{proof}
To avoid complicated expressions that would hide the key idea, we will remain at a rather qualitative level of description, leaving most calculations to the reader.

Given a quasi-caret $(\gamma_1, \gamma_2, \gamma_3)$, we can stack a sequence of disjoint consecutive balls along the $\gamma_j$'s, whose radii increase linearly with the distance to the center $v_0$. More precisely, one can find for every $j=1,2,3$ a sequence of balls $B^j_k=B(\gamma_j(i_k),r_k^j)$ such that
\begin{itemize}
\item
 $C^{-1}d(v_0,\gamma(i_k))\leq  r_k^j\leq Cd(v_0,\gamma(i_k))$ for some constant $C\geq 1$,
 \item $r_k^j \geq c'R$ for some $0 < c' < 1$ independent of $R$,
 \item $\sum_{j,k} r_k^j \geq \alpha R$ where $\alpha > 1$ is also independent of $R$,
\item the distance between $B^j_k$ and $B^j_{k+1}$ equals $1$, and all these balls are disjoint (when j and k vary).
\end{itemize}

Now, fix some $R\leq \epsilon D$, and consider a ball $B(x,R)$. It contains a quasi-caret of radius $R$.  This caret can be replaced by the balls described above, each one of them containing a quasi-caret of radius $r_k^j$. The sum of the radii of these carets is at least $\alpha R$. We can iterate this procedure within each ball, so that at the $k$-th iteration, we obtain a set of disjoint balls in $B(x,R)$ whose radii sum to at least $\alpha^k R$. Because the radius $R$ decreases by a factor no smaller than $c'$ each time, we can iterate $\log_{1/c'} R$ times. After $\log_{1/c'}R$ iterations, each ball still has positive radius, so $|B(x,R)| \geq\alpha^{\log_{1/c'}R} R= R^{1+\log_{1/c'}(\alpha)}$. Since $\log_{1/c'}(\alpha) > 0$, this proves the lemma. \end{proof}

Next we will show that under the assumption that no such caret exists, our graphs locally converge to a line. More precisely,

\begin{lemma}
Let $X_n$ be a roughly transitive sequence of graphs of diameter $D_n$ going to infinity whose carets are of length $o(D_n)$. Then  there exists $c>0$ such that  for $n$ large enough, any ball of radius $cD_n$ is contained in the $o(D_n)$-neighborhood of a geodesic segment.
\end{lemma}

\begin{proof}
By rough transitivity, it is enough to prove the lemma for some specific ball of radius $cD_n$.
Start with a geodesic $[x,y]$ of length equal to the diameter $D_n$. Let $z$ be the middle of this geodesic. We are going to show that the ball of radius $D_n/10$ around $z$ is contained in a $o(D_n)$-neighborhood of $[x,y]$.
If this was untrue, we would find a constant $c'$ such that $B(z,D_n/10)$ contains an element w at distance at least $c'D_n$ of $[x,y]$. Now pick an element $z'$ in $[x,y]$ minimizing the distance from w to $[x,y]$. The shortest path from $z'$ to $w$, together with the two segments of the geodesic starting from $z'$ form a caret of size proportional to $D_n$.
\end{proof}

If we knew that $X_n$ converges, and that the limit is homogeneous and compact, then this lemma would show that the limit is a locally a line, and thus must be $S^1$. If we knew that there was a large geodesic cycle in $X_n$, then Lemma~\ref{no_small_carets_bis} would show that all vertices are close to the cycle, which would also imply that the limit is $S^1$. However, we know neither of these two facts a priori, so the next lemma is necessary to complete the poof.

\begin{lemma}
Suppose $X_n$ has the property that for some $c>0$, and for $n$ large enough, any ball of radius $cD_n$ is contained in the $o(D_n)$-neighborhood of a geodesic segment. Then its scaling limit is $S^1$.
\end{lemma}
\begin{proof}
Let $x_1,....,x_k$ be a maximal $cD_n/10$-separated set of points of $X_n$, and let $B_j$ be the corresponding balls of radius $cD_n/100$.

We consider the graph $H_n$ whose vertices are labeled by the balls $B_j$ and such that two vertices are connected by an edge if the corresponding balls are connected by a path avoiding the other balls.

By maximality, for any $v \in X_n$, there is at least one $x_i$ in $B(v,cD_n/5)$. Let us consider a fixed $x_j$. In $B(x_j, cD_n)$, $X_n$ is well approximated by a line, so there are vertices $v_1$ and $v_2$ on each side of $x_j$ such that the balls of radius $cD_n/5$ around $v_1$, $v_2$, and $x_j$ are disjoint. Thus, there must be an $x_i$ on each side of $x_j$. Picking on each side the $x_i$ that is closest to $x_j$, we see that the corresponding balls are connected in $H_n$. Moreover since removing these two balls disconnects the ball $B(x_j,cD_n/100)$ from all other $x_i$'s, we see that the degree of $H_n$ is exactly two. Hence the graph $H_n$ is a cycle that we will now denote by  $\Z/k\Z$. To simplify notation, let us reindex the balls $B_j$ accordingly by $\Z/k\Z$.

Let us show that the graph $X_n$ admits a (simplicial) projection onto $\Z/k\Z$.
If we remove the balls $B_i$, we end up with a disjoint union of graphs $C_1,\ldots, C_k$, such that $C_j$ is connected to $B_j$ and $B_{j+1}.$ Let $V_j=B_j\cup C_j$. The graph $V_j$ connects to and only to $V_{j-1}$ and $V_{j+1}$. Hence we have a projection from $X_n$ to the cyclic graph $\Z/k\Z$ sending $V_j$ to the vertex $j$, and edges between $V_j$ and $V_{j+1}$ to the unique edge between $j$ and $j+1$.

Recall that the distance between two consecutive balls  $B_i$ and $B_{i+1}$ is at least $ cD_n/20$. Now take a shortest loop $\gamma=(\gamma(1),\ldots ,\gamma(m)=\gamma(0))$ in $X_n$ among those projecting to homotopically non trivial loops in the graph associated to $\Z/k\Z$.
Clearly this loop has length at least $ckD_n/20$ (since it passes through all balls $B_i$).
We claim moreover that it is a geodesic loop. Without loss of generality, we can suppose that $\gamma(0)$ starts in $B_0$ and that the next ball visited by $\gamma$ after $B_0$ is $B_1$. Observe that although $\gamma$ might exit some $B_i$ and then come back to it without visiting any other $B_j$, it {\it cannot} visit $B_i$, then go to $B_{i+1}$, and then back to $B_{i}$ (without visiting other balls). Indeed such a backtrack path could be replaced by a shorter path staying within $B_i$, contradicting minimality.
It follows that the sequence of $B_i$'s visited by $\gamma$ (neglecting possible repetitions) is given by $B_0, B_1\ldots B_k=B_{0}$; namely it corresponds to the standard cycle in $\Z/k\Z$.
The same argument implies that the sequence of balls visited by any geodesic joining two points in $X_n$ corresponds to a (possibly empty) interval in $\Z/k\Z$.

Now, suppose for sake of contradiction that $\gamma$ is not a geodesic. This means that there exists an interval of length $\leq m/2$ in $\gamma$ which does not minimize the distance between its endpoints. But then applying the previous remark, we see that replacing either this interval or its complement by a minimizing geodesic yields a loop whose projection is homotopically non-trivial, hence contradicting our minimal assumption on $\gamma$.

We therefore obtain a geodesic loop in $X_n$ whose Hausdorff distance to $X_n$ is
 in $o(D_n)$. Hence the scaling limit of $X_n$ exists and is isometric to $S^1$.
\end{proof}

\subsection{Proof of Theorem~\ref{theorem.main}}
\label{section.transitive}

In this section, we present an elementary proof of Theorem~\ref{theorem.main}.
To prove Theorem~\ref{theorem.main}, It suffices to show that in a finite vertex transitive graph with small volume relative to its diameter, there is a geodesic cycle whose length is polynomial in the diameter. If a long caret is rooted at a vertex on this cycle, then using transitivity and iteration we generate large volume, contradicting the assumption. Thus, all vertices must be close to the cycle.

To find a large geodesic cycle, we use the fact that a finite vertex transitive graph $X$ contains a $\diam(X)/8$-fat triangle. If this triangle is homotopic to a point after filling in small faces, this will imply large area and will violate our assumption. Thus, there is a loop that is not contractible. The smallest non-contractible loop is a geodesic cycle, and since we filled in all small cycles, this geodesic cycle must be large.

\medskip

We begin by proving a version of Lemma~\ref{no_small_carets_bis} for vertex transitive graphs.

\begin{definition}
A {\em 3-caret  of branch-length $R$} is a triple $\gamma_1$, $\gamma_2$, $\gamma_3$ of geodesics from a vertex $v_0$ to vertices $v_1$, $v_2$, and $v_3$, respectively, such that $d(v_0, v_i) = R$ for $i = 1,2,3$, and for all $k_1,k_2,k_3$, $d(\gamma_i(k_i),\gamma_j(k_j))\geq \max\{k_i,k_j\}$ for  $i\neq j$.
\end{definition}

\begin{lemma}
\label{no_small_carets}
Let $D = \diam(X)$. Suppose there exists a 3-caret of branch length $R = \epsilon D^c$ for some $\epsilon, c > 0$ in a finite vertex transitive graph $X$. Then $|X| > \epsilon' D^{1+c(\log_3(4)-1)}$, where $\epsilon' = (1/2) \epsilon^{\log_3(4)-1}$.
\end{lemma}

\begin{proof}

Suppose $\gamma_1$, $\gamma_2$, and $\gamma_3$ form a 3-caret of branch length $R$. Let $u_1$, $u_2$, and $u_3$ denote the vertices at distance $2R/3$ from $v_0$ on $\gamma_1$, $\gamma_2$, and $\gamma_3$, respectively, and let $u_0:= v_0$. The $u_i$ are at pairwise distance $2R/3$ from each other, and so $B(u_i,R/3)$ are pairwise disjoint. By vertex transitivity, there is a 3-caret of branch length $R$ centered at each $u_i$, which intersects $B(u_i, R/3)$ as a 3-caret of branch length $R/3$. Thus, we have four disjoint balls of radius $R/3$, each containing a 3-caret of radius $R/3$.

We can iterate this procedure, dividing $R$ by three at each step and multiplying the number of disjoint balls by four. So for any $m$, $B(v_0,R)$ contains $4^m$ balls, each of which contains a 3-caret of branch length $R/3^m$. Letting $m = \log_3(R)$, we have that $B(v_0,R)$ contains $4^m$ disjoint 3-carets of branch length 1. In particular, $|B(v_0,R)| \geq 4^m = R^{\log_3(4)}$.

There exists a geodesic path $\gamma$ in $X$ of length $D$. Let $R = \epsilon D^c$. Then it is possible to take vertices $v_1, \ldots, v_{D/2R}$ in $\gamma$ such that $B(v_i, R) \cap B(v_j, R) = \emptyset$ for all $i \neq j$. Summing the number of vertices in $B(v_i,R)$ for $1 \leq i \leq D/(2R)$, and using that $|B(v_i, R)| \geq R^{\log_3(4)}$, we have
 $$|X| \geq D/(2R) \cdot R^{\log_3(4)} = (1/2) \epsilon^{\log_3(4)-1} D^{1+c(\log_3(4)-1)}.$$

\end{proof}

The fact that a 3-caret of branch length $R$ implies $|B(v_0, R)| \geq R^{\log_3(4)}$ also has consequences for infinite vertex transitive graphs. For example, if an infinite vertex transitive graph $X$ has linear growth, then there is an upper bound on the size of a 3-caret in $X$. Since $X$ has a bi-infinite geodesic $\gamma$ and a vertex at distance $R$ from $\gamma$ implies a 3-caret of branch length $R$, every vertex in $X$ must be within a bounded neighborhood of $\gamma$. Conversely, if $X$ does not have linear growth, then there must be vertices at arbitrary distances from any fixed bi-infinite geodesic. Thus $X$ must have growth at least $O(n^{\log_3(4)})$.

Next, we will show that every vertex-transitive graph with a large diameter has a large geodesic cycle. We will use the following theorem from~\cite{BS}.

\begin{definition}
A geodesic triangle with sides $s_1$, $s_2$, $s_3$ is $\delta$-fat if for every vertex $v$ in $X$,
$$dist(v,s_1)+dist(v,s_2)+dist(v,s_3) \geq \delta.$$
\end{definition}

\begin{theorem}
\label{triangle_existence}
Every finite vertex transitive graph with diameter $D$ contains a $(1/8)D$-fat triangle.
\end{theorem}

For completeness, here is the short proof. Given vertices $u$ and $v$, let $uv$ denote a shortest path from $u$ to $v$

\begin{proof}
Suppose $X$ is finite and transitive, and $D$ is its diameter. Let $w$ and $z$ realize the diameter, i.e. $|wz|=D$. By transitivity there is a geodesic path $xy$ that has $z$ as its midpoint and length $D$. Suppose the triangle $wxy$ is not $\delta$-fat. Then there is a point $a$ on $xy$ such that the distance from $a$ to $wy$ is at most $2 \delta$ and the distance from $a$ to $wx$ is at most $2\delta$. Suppose, w.l.o.g. that $a$ is closer to $x$ than to $y$. We have $|ax|+|ay|=D$, $|wa|+|ax| < 2 \delta + D$ (because $a$ is within $2 \delta$ of $wx$), $|wa|+|ay|< 2 \delta + D$. Add these latter two and subtract the previous equality, and get $|wa| < D/2 + 2 \delta$. Since $|wz| = D$, this means that $|za| > D/2 - 2 \delta$. Since $a$ is on $xy$ and closer to $x$, this means that $|xa| < 2 \delta$. Since $a$ is within $2 \delta$ from $wy$, we have $|wy| > |wa| + |ay| - 2 \delta$. Since $|xa| < 2 \delta$ and $|xy| = D$ this gives $|wy| > |wa| + D - 4 \delta$. Since $|wy|$ is at most $D$, this implies $|wa| < 4 \delta$. But $|za|$ is at most $D/2$. so $D = |wz| \le |wa| + |za| < 4 \delta + D/2$ So $D < 8 \delta$.
\end{proof}

\begin{lemma}
\label{cycle_existence}
Suppose $X$ is a finite $d$-regular vertex-transitive graph such that $|X| < (\alpha/d)D^{2-c}$, where $\alpha = \sqrt{3}/576$. Then $X$ contains a geodesic cycle of length $D^c$.
\end{lemma}

\begin{proof}
We will begin by proving two claims.\\

\noindent {\em Claim 1: Suppose $H$ is a $d$-regular planar graph, every face of $H$ except the outer face has a boundary of length at most $D^c$, and $H$ contains a $(1/8)D$-fat geodesic triangle. Then $|H| > (\alpha/d) D^{2-c}$.}

This is a variant of Besicovich' lemma for squares. Fill each face $f$ of $H$ with a simply connected surface of area at most $|f|^2$ so that distances in $H$ are preserved (for example, a large portion of a sphere), and consider the $(1/8)D$-fat geodesic triangle in the simply connected surface $X$ obtained. The triangle has sides $s_1$, $s_2$, and $s_3$, of lengths at least $(1/8)D$. The map $f$ from $X$ to $\R_+^3$ taking a point $x$ to $(dist(x,s_1), dist(x,s_2), dist(x,s_3))$ is 3-Lipschitz, so the area of the image of $f$ is smaller than $9$ times the area of $X$. For each $(x_1, x_2, x_3)$ in the image we have $x_1 +x_2 +x_3 > (1/8)D$, so projecting radially to the simplex $x_1 +x_2 +x_3 = (1/8)D$ does not increase the area. The projection of the image of the boundary of the triangle is the boundary of the simplex, so the projection is onto. Thus, the area of $X$ is bigger than $1/9 \sqrt{3} ((1/8)D)^2$.

Each face $f$ of $H$ contributes $|f|^2$ to the area of $X$, so we can say that each vertex on the border of $f$ contributes $|f|$ to the area of $X$. Each vertex of $H$ participates in at most $d$ faces, each of which is of size at most $D^c$ so $\area(X) \leq  dD^c |H|$. Thus $|H| > \alpha/d D^{2-c}$ where $\alpha =(1/9) \sqrt{3}(1/8)^2$.\\

\noindent {\em Claim 2:  Suppose $X$ is a finite $d$-regular graph that contains a $(1/8) D$-fat triangle $A$, and let $T$ denote the topological space obtained from $X$ by replacing each cycle of length at most $D^c$ with a Euclidean disc whose boundary matches the cycle. If $A$ is homotopic in $T$ to a point, then $|X| > (\alpha/d)D^{2-c}$.}

Suppose $A$ is homotopic to a point. Then a continuous map from $S^1$ to $A$ can be extended to a continuous map from the disk $B^1$ to $T(X,D^c)$. The image of this map has a planar sub-surface $S$ with boundary $A$. Intersecting $S$ with $X$, we obtain a planar subgraph $H$ of $X$ such that each face has a boundary of length at most $D^c$ except for the outer face, which has $A$ as a boundary. By Claim 1, $|X| \geq |H| > (\alpha/d) D^{2-c}$.\\

Now we will prove the lemma. Because $|X| < (\alpha/d) D^{2-c}$, Claim 2 tells us that $T$ is not simply connected. Any loop in $T$ is homotopic to a loop in $X$, so since $T$ is not simply connected, there exists a topologically non-trivial loop in $X$. Let $\ell$ denote the non-trivial loop in $X$ which has minimal length.

Given any two vertices $u$ and $v$ in $\ell$, the shortest path from $u$ to $v$ is homotopic to at most one of the two paths $p_0$, $p_1$ in $\ell$ between $u$ and $v$; say it is not homotopic to $p_0$. If the length of $p^*$ is shorter than the minimum length of $p_0$ and $p_1$, then it would be possible to replace $p_1$ with $p^*$ to obtain a loop $\ell'$ which is non-trivial and shorter than $\ell$. Thus, the shortest path in $X$ between any two vertices in $\ell$ is a path in $\ell$, so $\ell$ is a geodesic cycle.

Because all cycles of length less than $D^c$ are homotopic to a point in $T$, $\ell$ must have length greater than $D^c$.
\end{proof}

\begin{proof}[Proof of Theorem~\ref{theorem.main}]
Let $D_n$ denote the diameter of $X_n$ and $c = \frac{1}{\log_3(4)}.$ It suffices to show that for large enough $n$, there is a geodesic cycle $C_n$ in $X_n$ such that $|C_n| > {D_n}^c$ and
$$\max_{v \in X_n}(dist(v,C_n)) = o({D_n}^c).$$

For large enough $n$, $|X_n| < (\alpha/d){D_n}^{2-c}$, so Lemma~\ref{cycle_existence} guarantees the existence of a geodesic cycle $C_n$ with $|C_n| > {D_n}^c$. By Lemma~\ref{no_small_carets}, if there were a 3-caret of branch length $\epsilon {D_n}^c$, we would have $|X_n| \geq \epsilon' {D_n}^{1+c(\log_3(4)-1)} = \epsilon' {D_n}^{2-c}$. So for all $\epsilon$ and large enough $n$, there is no 3-caret of branch length $\epsilon {D_n}^{c}$. But for $\epsilon < {D_n}^c/4$, a vertex at distance $\epsilon {D_n}^c$ from $C_n$ implies a 3-caret of branch length $\epsilon {D_n}^c$. Thus, for every $\epsilon$ and for large enough $n$, all vertices in $X_n$ are within distance $\epsilon {D_n}^c$ from $C_n$.
\end{proof}

\subsection{A counterexample to  Theorem \ref{theorem.bis} in higher dimensions}\label{sec:counterexamples}
We will see in this section that Theorem \ref{theorem.bis} cannot be generalized to ``higher" dimensions (as  in Theorem~\ref{theorem:Main}). First of all, we will prove that in some sense, any compact manifold can be approximated by a roughly transitive sequence of graphs. Moreover, we will show that there are sequences of roughly transitive graphs with no converging subsequence, although with a polynomial control on the volume. 

Let us start with a useful lemma 

\begin{lemma}\label{lem:discretization}
Let $(X,d_X)$ be a compact geodesic metric space, let $t>0$ and let $D$ be a maximal subset of $X$ such that two distinct elements are at distance at least $t$. Let $G$ be the graph whose set of vertices is $D$ and such that two elements of $D$ are joined by an edge if they lie at distance at most $4d$. Finally equip $D$ with the graph metric $d$ whose edges have length $4t$.
Then the inclusion $\phi:(D,d)\to (X,d_X)$ si $1$-Lipschitz and a $(4,t)$-quasi-isometry.
\end{lemma}
\begin{proof}
Two neighbors in $D$ being at distance at most $4t$ in $X$, it follows that $\phi$ is $1$-Lipschitz. By maximality of $D$, every element in $X$ is at distance at least $t$ from some element of $D$: hence the additive constant $t$. We are left to prove that given two points $x, y$ in $D$ at distance $d_X(x,y)\leq nt$ for some $n\in\{1,2,\ldots\}$, one has $d(x,y)\leq 4nt$. Consider a minimizing geodesic between $x$ and $y$ and let $x=x_0, x_1 \ldots x_n=y$ be points on this geodesic such that $d_X(x_i,x_{i+1})=t$. For every $i=1, \ldots, n-1$, one can find an element $d_i\in D$ at distance at most $t$ from $x_i$. It follows that the sequence $x=d_0, d_1,\ldots, d_{n-1},d_n=y$ is such that $d_X(d_i,d_{i+1})\leq 3t$ and therefore are neighbors in $D$. We deduce that $d(x,y)\leq 4nt$. 
\end{proof}

\begin{corollary}\label{cor:approx}
For every compact metric geodesic space $(X,d_X)$, there exists a sequence of rescaled graphs $X_n$ converging in the Gromov Hausdorff sense to $(X,d')$, where $d'$ is $4$-bi-Lipschitz equivalent to $d$.
\end{corollary}
\begin{proof}
By the previous proposition, one has a sequence $X_n$ of rescaled graphs which is $(4,1/n)$-quasi-isometric to $X$. Recall that by Gromov's compactness criterion (see Lemma \ref{lemPrelim:compact}) that a sequence of compact metric spaces is relatively Hausdorff Gromov compact if and only if the $X_n$'s have bounded diameter, and are ``equi-relatively compact": for every $\eps>0$, there exists $N\in \N$ such that $X_n$ can be covered by at most $N$ balls of radius $\eps.$ Note that the content of Lemma \ref{lem:discretization} is that $X_n$ can be chosen to be $4$-bi-Lipschitz embedded into $X$. Equi-compactness of the $X_n$'s therefore follows from the fact that $X$ is compact. This implies that up to passing to a subsequence, $X_n$ converges to some compact metric space. The last statement is obvious.    
\end{proof}

\begin{definition}
A compact metric space $X$ is called uniformly bi-Lipschitz transitive if there exists $L\geq 1$ such that for any two points $x$ and $y$ in $X$, there exists an $L$-bi-Lipschitz homeomorphism of $X$ mapping $x$ to $y$.
\end{definition}

\begin{lemma}\label{lem:biliptransitive}
A compact Riemannian manifold is uniformly bi-Lipschitz transitive. Moreover, if $M_n$ is a sequence of Galois covers of $M$ (e.g.\ the universal cover), then the sequence $M_n$ is uniformly bi-Lipschitz transitive (that is, uniformly with respect to $n$). 
\end{lemma}
\begin{proof}
Actually, one has better: for every pair of points $x,y$, there exists a global diffeomorphism $f$ mapping $x$ to $y$ such that the norms of the derivatives of $f$ and its inverse are bounded by $L$. Let us briefly sketch the argument for $M$. By compactness, it is enough to show that there exist $\eps>0$ such that the previous statement holds for pair of points at distance at most $\eps$: this is an easy exercise in Riemannian geometry.

Let us prove the second statement (which will be used later on in this section). Let $M'$ be a Riemannian manifold and let $G$ be a group acting properly discontinuously and freely by isometries on $M'$ such that the quotient of $M'$ by this action is isometric to $M$. Denote $\pi:M'\to M$ the covering map. Let $x,y\in M'$ and let $f$ be an $L$-bi-Lipschitz homeomorphism of $M$ sending $\pi(x)$ to $\pi(y)$. Lifting $f$ to $M'$ yields an $L$-bi-Lipschitz homeomorphism $f'$ of $M'$ sending $x$ to some element in the $G$-orbit of $y$. Now composing $f'$ with some $g\in G$, we obtain some $L$-bi-Lipschitz homeomorphism sending $x$ to $y$. 
\end{proof}
Now notice that an approximation $X_n$ of $X$ as in the proof of Corollary \ref{cor:approx} is roughly transitive as soon as $X$ is uniformly bi-Lipschitz transitive. We therefore obtain

\begin{proposition}
Given a compact Riemannian manifold $(M,d_M)$, there exists a roughly transitive sequence of graphs converging to $(M,d')$ where $d'$ is bi-Lipschitz equivalent to $d_X$. 
\end{proposition}

As announced at the beginning of this section, we also have
\begin{proposition}\label{prop:nonlimit}
For every positive sequence $u(n)$ such that $u(n)/n^2\to \infty$, there exists a roughly transitive sequence of finite graphs $X_n$ such that $|X_n|\leq u(\diam(X_n)$ for large $n$, and such that no subsequence of $X_n$ has a scaling limit. 
\end{proposition}
\begin{proof} We start by picking a sequence of Cayley graphs $Y_n$ such that no subsequence has a rescaled limit (but without polynomial control on the volume). For example, let $S$ be a finite generating subset of $G= \SL(3,\Z)$, and let $Y_n$ be the Cayley graph of $G_n=\SL(3,\Z/n\Z)$ associated to the (projected) generated set $S$. The fact that this is a sequence of expanders follows from the fact that $G= \SL(3,\Z)$ has property $T$ (see for instance \cite[Section 2]{Lu}).

\begin{lemma}\label{lem:expander}
The sequence $\tilde{Y_n}$ obtained by rescaling $Y_n$ is not relatively compact (nor is any of its subsequences).  
\end{lemma}
\begin{proof}
Recall that the expanding condition says that there exists $c>0$ such that every $A\subset X_n$ with $|A|\leq |X|/2$ satisfies  $|\partial A|\geq c|A|$, where $\partial A$ is the set of vertices outside $A$ but at distance $\leq 1$ from $A$. By Gromov's compactness criterion, if $\tilde{Y_n}$ is relatively compact, 
there exists some $N\in \N$ such that $Y_n$ can be covered by at most $N$ balls of radius $R_n=\diam(Y_n)/10$. But this implies that 
\begin{equation}\label{eq:VolumeDoubling}
|Y_n|\leq N|B(R_n)|.
\end{equation} 
On the other hand by the expanding condition, the ball of radius $2R_n$ is of size at least $(1+c)^{R_n}|B(R_n)|$, and since $R_n\to \infty$, this contradicts (\ref{eq:VolumeDoubling}).
\end{proof}
Now we need to replace $Y_n$ by a polynomially growing sequence without loosing either the fact that it has no converging rescaled subsequence,  nor its rough transitivity. A first  idea to get a polynomially growing sequence is to elongate the edges: this indeed can get any growth (even close to linear), but destroys the rough transitivity. Instead, we convert $G_n$-equivariantly $Y_n$ into a Riemannian surface $S_n$ by replacing edges with empty tubes, and smoothing the joints that correspond to vertices in $Y_n$. Both the length and the radius of the tubes in $S_n$ will be equal to $L_n$ to be determined later. Observe that $G_n$ acts freely by isometries on $Y_n$ and that up to rescaling $S_n/G_n$ by $L_n$, they are all isometric to a same surface, which looks like an wedge of $d$ circles (where $d$ is the degree of $Y_n$), where each circle has been replaced by a tube. We can now invoke Lemma \ref{lem:biliptransitive} to see that the sequence is uniformly bi-Lipschitz transitive.   
To obtain the sequence $X_n$ of roughly transitive graphs, discretize $S_n$ as in Lemma \ref{lem:discretization}, with $t=1$.

To summarize, we have constructed a sequence $X_n$ whose rescaled sequence is $(O(1),o(1))$-quasi-isometric to $\tilde{Y_n}$ which we know has no converging subsequence. It follows that $X_n$ itself has no subsequence with a rescaled limit. Now the interesting fact is that the sequence $X_n$ depends on a parameter, namely the sequence $L_n$, which we will now adjust in order to control the growth of $X_n$.
Note that the cardinality of $X_n$ is of the order of $|Y_n|L_n^2$, while its diameter $D_n$ is roughly $Diam(Y_n)L_n$. Hence letting $L_n$ grow sufficiently fast, it is possible to obtain that $|X_n| = O(u(\diam(X_n))$. 
\end{proof}

\section{Open problems}

\subsection{Generalizing Theorems \ref{theorem:Main} and \ref{theorem:Main'} to infinite graphs}\label{section:discussionInfinite}

\begin{conjecture}\label{conj}
let $ D_n\leq \Delta_n$ be two sequences going to infinity, and let $(X_n)$ be a sequence of vertex transitive graphs such that the balls of radius $D_n$ satisfy $|B(x,D_n)|=O(D_n^q)$. Then $(X_n,d/\Delta_n)$ has a subsequence converging for the pointed GH-topology to a connected nilpotent Lie group equipped with a Carnot-Caratheodory metric. 
\end{conjecture}
In \S \ref{section:Main}, we proved this conjecture under the unpleasant assumption that the isometry group of $X_n$ admits a discrete (hence finitely generated) group acting transitively. Observe that this assumption is automatically satisfied if the graphs are finite.
The existence of a discrete subgroup of the isometry group acting transitively is needed for a crucial step of the proof: the reduction to Cayley graphs of nilpotent groups. This step  relies  on the main result of \cite{BGT} (see Theorem \ref{BGT}). In order to be able to prove a full version of the above conjecture, one would need a locally compact version of their result. 


\subsection{What about local convergence?}
Let $D_n$ be a sequence going to infinity, and let $(X_n)$ be a sequence of vertex transitive graphs with the same degree such that the balls of radius $D_n$ satisfy $|B(x,D_n)|=O(D_n^q)$. Since these graphs have bounded degree, they form a relatively compact sequence for the pointed GH-topology. So assuming that it converges to $X$, what can be said about $X$ (except for the fact that it is transitive, which is due to \cite{FY})? In view of Conjecture \ref{conj}, one may expect the asymptotic cone of $X$ to be a nilpotent Lie group. This would follow if $X$ itself had polynomial growth. But this turns out to be false: indeed, we shall see in Remark \ref{rem} below that $X$ can be for instance a $k$-regular tree, for  some $k\geq 3$.

In this section we shall address a similar question for marked groups. For background on marked group topology, see \cite{Cham}. Recall that a $k$-marked group $(G,T)$ is the data of a group $G$, together with a $k$-tuple of elements $T=(t_1,\ldots g_k)$  forming a generating set of $G$. In other words a $k$-marked group is a quotient of $F_k$ with the choice of an epimorphism $F_k\to G$. 
By a ``word" $w$, we shall mean simultaneously an element in $F_k$ and the corresponding reduced word in $T$.  Let us start with a trivial  lemma.

\begin{lemma}
 Let $(G_n,S_n)$ be a sequence of marked groups converging to $(G,S)$. Let $(Q_n,T_n)$ be a quotient of $(G_n,S_n)$ and assume that $(Q_n,T_n)$ converges to $(Q,T)$, then this limit is a quotient of $(G,S)$ 
\end{lemma}
\begin{proof} Indeed, all relations of $G$ are eventually satisfied by $G_n$, so $Q$ satisfies all relations of $G$. 
\end{proof}

\begin{proposition}\label{prop:locallimit} 
Let $q>0, k\in \N$, and let $D_n\to \infty$. Let $(G_n,S_n)$ be a sequence of $k$-marked groups converging to some $(G,S)$, and such that $|S_n^{D_n}|=O(D_n^q)$. If  in addition $D_n\leq \diam(G_n,S_n)$, then $G$ admits an infinite virtually nilpotent quotient. 
\end{proposition}
\begin{proof}
Note that by Theorem \ref{BGT}, there exist two positive integers $i$ and $l$ such that $G_n$ has a subgroup $G_n'$ of index $i$, and a normal subgroup $F_n$ such that $G_n'/F_n$ is nilpotent of step $\leq l$, and such that $F_n$ has diameter small in $o(D_n)$. The condition $D_n\leq \diam(G_n,S_n)$ therefore implies that the quotient $G_n/F_n$ has diameter going to infinity. 

Let us assume $G'_n$ is normal in $G_n$, and up to taking a subsequence, we can assume that $(G_n,S_n)/G'_n$ is some fixed finite marked group $(K,U)$. Since $K$ is finite, it is the quotient of the free group $\langle S\rangle$ by some finitely generated normal subgroup $N$. Let us take a finite set $\mathcal{W}$ generating $N$ as a subgroup, and let $\mathcal{R}$ be the set of all iterated commutators (see \S \ref{subsection:prelim3}) of length $l$ in $\mathcal{W}$. 
Clearly the quotient $Q_n$ of $G_n$ by the normal subgroup generated by $\mathcal{R}$ has a nilpotent subgroup of step $\leq l$ and of index $|K|$ that surjects to $G'_n/F_n$.  
Therefore every cluster point of $(Q_n)$ is infinite and virtually nilpotent. We conclude thanks to the lemma.
\end{proof}

\begin{remark}\label{rem}
For the sake of illustration, let us give an example where the limit in Proposition \ref{prop:locallimit} is the standard Cayley graph of the free group with two generators $(F_2,S)$. Take any sequence of finite Cayley graphs $(G_n,S_n)$ converging to  $(F_2,S)$. Now let $p_n$ be a sequence of primes such that $p_n\geq |G_n|+1$. Consider the sequence of groups $G_n'=G_n\times \Z/p_n\Z$, and for each $n\in \N$, let $a_n$ be a generator of $ \Z/p_n\Z$. Now consider the generating subset $S_n'=\{sa_n,\; s\in S_n\}$ of $G_n'$. Clearly, the diameter of $(G_n',S_n')$ is at least $p_n$, hence one has $|G_n'|\leq \diam(G_n')^2$. On the other hand, it is easy to check that $(G_n',S_n')$ converges to $(F_2,S)$. 
\end{remark}

\subsection{Roughly transitive graphs and their limits}
Section \ref{sec:counterexamples} leaves several open questions:
\begin{itemize}
\item Is there a sequence of roughly transitive graphs with at most quadratic growth which does not admit a converging subsequence?
\item What is the supremum of all $\delta$ for which Theorem \ref{theorem.bis} holds? Does it actually  depend on the rough-transitivity constant? Note that Proposition \ref{prop:nonlimit} shows that it is $\leq 2$.
\item Does there exists a compact finite dimensional geodesic metric uniformly bi-Lipschitz transitive which is not a manifold? Note that a {\it converging} counter-example to Theorem \ref{theorem.bis} with $\delta<2$ would yield such an example. Compact metric groups could be thought as a possible source of examples, (for instance $\Z/2\Z^{\N}$ is a Cantor set). However, a geodesic, finite dimensional compact group is automatically a Lie group \cite{PW}. On the other hand, we do not know whether the Menger curve, the Sierpi\'nski carpet, or any other  fractal sets admit geodesic distances for which they are uniformly bi-Lipschitz transitive.
\end{itemize}

\medskip

\noindent
{\bf Acknowledgements:} We are grateful to Pierre Pansu for useful discussions and especially to Emmanuel Breuillard for directing us to Theorem \ref{BGT} and also for helping us extend our result to sequences without bounded degree.  
We would also like to  thank Tsachik Gelander for pointing us to Turing's theorem, and its applications to homogeneous manifolds. We are grateful to L\'aszl\'o  Pyber for noticing a mistake in a previous version of Lemma \ref{lem:Torsion Nilpotent}. He also mentioned to us some work in progress where some form of Lemma \ref{lem:Torsion Nilpotent} is proved. Last but not least, we would like to thank the three referees for their numerous remarks and corrections.

\appendix

\end{document}